\documentclass[final,12pt,a4paper,sort&compress]{elsarticle}
\usepackage[dvipsnames]{xcolor}
    	\definecolor{AR}{rgb}{0.64, 0.0, 0.0}
\usepackage{float}
\newcommand{\dott}{\, \cdot\,}
\usepackage{eurosym}
\usepackage{subcaption}
\usepackage[export]{adjustbox}
\usepackage{tikz}
\DeclareRobustCommand\full  {\tikz[baseline=-0.6ex]\draw[thick] (0,0)--(0.5,0);}
\DeclareRobustCommand\dotted{\tikz[baseline=-0.6ex]\draw[thick,dotted] (0,0)--(0.54,0);}
\DeclareRobustCommand\dashed{\tikz[baseline=-0.6ex]\draw[thick,dashed] (0,0)--(0.54,0);}

\DeclareRobustCommand\chainn{\tikz[baseline=-0.6ex]\draw[thick,dash dot] (0,0)--(0.5,0);}
\usepackage{tikz}
\usepackage{capt-of}
\newcommand{\nn}{\nonumber}

\newcommand{\D}{\Delta}

\newcommand{\BV}{BV}

\usepackage{bbm}
\usepackage{subcaption}
\usepackage{amssymb}
\usepackage{mathrsfs}
\usepackage{mathtools}
\usepackage{amsmath}
\usepackage{xcolor}
\definecolor{darkgreen}{rgb}{0.0, 0.2, 0.13}

\newcommand{\norma}[1]{{\left\|#1\right\|}}

\usepackage{cases}
\usepackage{amsfonts}
\usepackage{amssymb}
\usepackage{amsthm}
\usepackage{graphicx}
\usepackage{mathrsfs}
\usepackage{xcolor}

\usepackage{latexsym}
\usepackage{tikz}
\usepackage{lineno}
\usepackage{caption}

\usepackage[colorlinks,plainpages=true,pdfpagelabels,hypertexnames=true,colorlinks=true,pdfstartview=FitV,linkcolor=black,citecolor=black,urlcolor=black]{hyperref}
\PassOptionsToPackage{unicode}{hyperref}
\PassOptionsToPackage{naturalnames}{hyperref}

\usepackage{enumerate}
\usepackage[shortlabels]{enumitem}
\usepackage{bookmark}

\usepackage{esint}

\usepackage[ddmmyyyy]{datetime}

\usepackage[margin=2cm]{geometry}
\makeatletter
\g@addto@macro\normalsize{%
  \setlength\abovedisplayskip{4pt}
  \setlength\belowdisplayskip{4pt}
  \setlength\abovedisplayshortskip{4pt}
  \setlength\belowdisplayshortskip{4pt}
}
\numberwithin{equation}{section}
\everymath{\displaystyle}

\usepackage[capitalize,nameinlink]{cleveref}
\crefname{section}{Section}{Sections}
\crefname{subsection}{Subsection}{Subsections}
\crefname{condition}{Condition}{Conditions}
\crefname{hypothesis}{Hypothesis}{Conditions}
\crefname{assumption}{Assumption}{Assumptions}
\crefname{lemma}{Lemma}{Lemmas}
\crefname{definition}{Definition}{Definitions}

\crefformat{equation}{\textup{#2(#1)#3}}
\crefrangeformat{equation}{\textup{#3(#1)#4--#5(#2)#6}}
\crefmultiformat{equation}{\textup{#2(#1)#3}}{ and \textup{#2(#1)#3}}
{, \textup{#2(#1)#3}}{, and \textup{#2(#1)#3}}
\crefrangemultiformat{equation}{\textup{#3(#1)#4--#5(#2)#6}}%
{ and \textup{#3(#1)#4--#5(#2)#6}}{, \textup{#3(#1)#4--#5(#2)#6}}%
{, and \textup{#3(#1)#4--#5(#2)#6}}
\Crefformat{equation}{#2Equation~\textup{(#1)}#3}
\Crefrangeformat{equation}{Equations~\textup{#3(#1)#4--#5(#2)#6}}
\Crefmultiformat{equation}{Equations~\textup{#2(#1)#3}}{ and \textup{#2(#1)#3}}
{, \textup{#2(#1)#3}}{, and \textup{#2(#1)#3}}
\Crefrangemultiformat{equation}{Equations~\textup{#3(#1)#4--#5(#2)#6}}%
{ and \textup{#3(#1)#4--#5(#2)#6}}{, \textup{#3(#1)#4--#5(#2)#6}}%
{, and \textup{#3(#1)#4--#5(#2)#6}}

\crefdefaultlabelformat{#2\textup{#1}#3}
\numberwithin{equation}{section}

\newtheorem{theorem} {Theorem}[section]

\newtheorem{lemma}{Lemma}[section]

\newtheorem{counter example}{Counter Example}[section]

\newtheorem{remark} {Remark}[section]

\newtheorem{definition} {Definition}[section]

\def\N{\mathbb{N}}
\def\CC{{\rm \kern.24em \vrule width.02em height1.4ex depth-.05ex \kern-.26emC}}

\def\TagOnRight

\def\AA{{it I} \hskip-3pt{\tt A}}

\def\QQ{\rlap {\raise 0.4ex \hbox{$\scriptscriptstyle |$}} {\hskip -0.1em Q}}

\makeatletter
\newcommand{\vo}{\vec{o}\@ifnextchar{^}{\,}{}}
\makeatother

\def\YYint#1#2#3{{\setbox0=\hbox{$#1{#2#3}{\iint}$}
    \vcenter{\hbox{$#2#3$}}\kern-.50\wd0}}

\def\XXint#1#2#3{{\setbox0=\hbox{$#1{#2#3}{\int}$}
    \vcenter{\hbox{$#2#3$}}\kern-.50\wd0}}

\makeatletter
\def\namedlabel#1#2{\begingroup
   \def\@currentlabel{#2}%
   \label{#1}\endgroup
}
\makeatother

\makeatletter
\newcommand{\rmh}[1]{\mathpalette{\raisem@th{#1}}}
\newcommand{\raisem@th}[3]{\hspace*{-1pt}\raisebox{#1}{$#2#3$}}
\makeatother

\newcommand{\descref}[2]{\hyperref[#1]{\textnormal{\textcolor{black}{(}\textcolor{black}{\bf #2}\textcolor{black}{)}}}}

\newcommand{\dref}[2]{\hyperref[#1]{\textcolor{black}{(}\textcolor{black}{\bf #2}\textcolor{black}{)}}}

\newcommand{\be} {\begin{align}}
\newcommand{\ee} {\end{align}}
\newcommand{\Bea} {\begin{align*}}
\newcommand{\Eea} {\end{align*}}

\newcommand{\la} {\lambda}

\newcommand{\R}{\mathbb{R}}

\newcommand{\sgn}{\mathop\mathrm{sgn}}

\renewcommand{\L}[1]{\mathbf{L^#1}}

\newcommand{\Z}{\mathbb{Z}}

\DeclareMathOperator{\loc}{loc}
\DeclareMathOperator{\lipR}{Lip(\R)}
\DeclareMathOperator{\lip}{Lip}

\newcommand{\abs}[1]{\left| #1\right|}

\newcounter{whitney}
\refstepcounter{whitney}

\newcounter{ineqcounter}
\refstepcounter{ineqcounter}
\makeatletter
\def\ps@pprintTitle{%
\let\@oddhead\@empty
\let\@evenhead\@empty
\def\@oddfoot{}%
\let\@evenfoot\@oddfoot}
\makeatother
\usepackage[titletoc,toc,page]{appendix}

\usepackage{pgf,tikz}

\allowdisplaybreaks[4]
\renewcommand{\epsilon}{\varepsilon}
\renewcommand{\L}[1]{L^#1}

\newcommand{\modulo}[1]{{\left|#1\right|}}

\newcommand{\interi}{{\mathbb{Z}}}

\renewcommand{\d}[1]{\mathinner{\mathrm{d}{#1}}}

\usetikzlibrary{arrows}
\usetikzlibrary{decorations.pathreplacing}

\begin{document}
\begin{frontmatter}

\title{Convergence of {the numerical  approximations} and well-posedness:
Nonlocal conservation laws with rough flux}

\author[myaddress1]{Aekta Aggarwal}
\ead{aektaaggarwal@iimidr.ac.in}

\address[myaddress1]{Operations Management and Quantitative Techniques, Indian Institute of Management,\\Prabandh Shikhar, Rau--Pithampur Road, Indore, Madhya Pradesh 453556, India}

 \address[myaddress2]
 {Department of Mathematical Sciences, 
Norwegian University of Science and Technology,\\ NO–7491 Trondheim, Norway.}

\author[myaddress2]{Ganesh Vaidya}
\ead{ganesh.k.vaidya@ntnu.no}

\begin{abstract}
{We study} a class of { nonlinear} nonlocal conservation laws with discontinuous flux, modeling crowd dynamics and traffic flow, without any additional conditions on finiteness/discreteness of the set of
discontinuities or on the monotonicity of the kernel/the discontinuous coefficient. Strong compactness of the Godunov and Lax-Friedrichs type approximation{s} is proved, providing the existence of entropy solutions.
A proof of {the} uniqueness of the adapted entropy solution{s} is provided, establishing the convergence of the entire sequence of finite volume approximations to the adapted entropy solution. As per the current literature, this is the first well-posedness result for { the aforesaid class} and
connects the theory of nonlocal conservation laws (with discontinuous flux), with its local counterpart in a {generic} setup. Some numerical examples are presented to
display the performance of the schemes and explore the {l}imiting behavior of {these} nonlocal conservation law{s} to {their} local counterpart{s}.
\end{abstract}


\end{frontmatter}

\section{Introduction}

 One of the most widely used macroscopic models for traffic flow, due to its simplicity and rich analytical theory, is the
Lighthill-Whitham-Richards (LWR) model, given by \begin{eqnarray}\label{LWR}
u_t+(u\nu(u))_x=0,
\end{eqnarray} where $u$ is the average traffic density and $\nu$ is  the average traffic speed. However, the governing equation { follows} a { nonlinear} scalar hyperbolic conservation law, the solutions may develop discontinuities, resulting in infinite accelerations, {and thereby} limiting the model's capability to capture some important physical traffic phenomena. Over recent years, attention has thus shifted to conservation laws with \textit{nonlocal} parts in the flux, where a convolution is introduced in the flux to produce Lipschitz-continuous speeds in $x$ and $t$ ensuring
bounded accelerations. 

One of the ways to do this (see \cite{ACG2015,AG2016,BG2016,CGL2012,KP2017, KPS2018}), is to evaluate the  traffic speed from the average traffic density in the following way: \begin{eqnarray}
u_t+\left(u\nu(\mu*u)\right)_x&=&0,\label{134}
\end{eqnarray}where $\nu,
\mu \in (C^2 \cap W^{2,\infty}) (\R),$ with \[(\mu *\; u) (t,x)=
  \displaystyle\displaystyle \int\limits_{\R}\mu (x-\xi) \; u(t,\xi) \, \d\xi,\]
 where the mean traffic speed  resulted from the evaluation of the average traffic density.
Another natural nonlocal averaging (see \cite{FGKP2022, FKG2018}) is to average over the velocity resulting in the following  conservation law:
\begin{align}
u_t+\left(u(\mu*\nu(u))\right)_x&=0.\label{135}
\end{align}

It thus {is} interesting to consider the general class of such PDEs, namely
\begin{eqnarray}
 \label{eq:nl}
  \partial_t u
  +
  \partial_x ({f(u)} \nu(\mu*\overline{\nu}(u)))&=& 0\,\quad\quad \quad \quad \quad \text{for}\,\,\,(t,x) \in Q_{{T}}:=(0,{T})\times \R,
  \\ \label{IVP:data}
  u(0,x)&=&u_0(x) \quad\quad \quad \,\, \text{for}\,\,\,x \in \R.
\end{eqnarray}  
with $\nu,
\overline{\nu}\in (C^2 \cap W^{2,\infty}) (\R), { f\in \operatorname{Lip}(\R)},$ and ${f(0)=\nu(0)=\overline{\nu}(0)=0.}$
 The \emph{nonlocal} nature
of~\eqref{IVP:data} is invoked by $ \nu(\mu*\overline{\nu}(u))$ and is particularly suitable in describing the behavior of
crowds, where each member moves according to {thier} evaluation of the
crowd density and its variations within the horizon. The research area on nonlocal dynamics and in particular on nonlocal conservation laws has been gaining interest in the mathematical community over the last decade, for two major reasons, one that they are potential models for dynamics of multi agent systems like crowds \cite{CGL2012,CHM2011, CL2011}, traffic \cite {BG2016,CHL2011,BCV+2023,CKR2023, FGKP2023, LEE2019,GR2019}, opinion formation, sedimentation
models \cite{BBKT2011}, structured population dynamics 
 \cite{PER2007}, {laser technology \cite{CM2015}}, supply chain models \cite{CHM2011}, granular material dynamics \cite{AS2012} and  conveyor belt dynamics \cite{GHS+2014}, and the other being that the existing techniques of local counterparts are not directly applicable nor easily extendable to these PDEs. The uniqueness of the weak solutions for this class with a linear $f$, {without the use of any additional entropy condition}, has been settled in \cite{KP2021}, the existence of the solutions via the convergence of finite volume approximations has been studied in \cite{ ACG2015,AG2016,ACT2015,BG2016,FGKP2022}, {and the convergence rate analysis of the finite volume schemes has been recently studied in \cite{AHV2023}}.

 However, the roads can possibly be non uniform and rough (changing abruptly), motivating us to consider,
\begin{eqnarray}\label{13456}u_t+(F(x,u){\nu}(\mu*(\overline{\nu}(u)))_x=0\end{eqnarray} where $ F(x,u)$  can become a discontinuous function of $x$, with the set of discontinuities being possibly non discrete. The \textit{local} counterpart of \eqref{13456}, \begin{eqnarray}\label{134586}u_t+(F(x,u)\nu(\overline{\nu}(u)))_x=0,\end{eqnarray} 
  is quite rich by now, {a very incomplete list of references being}, \cite{AMV2005,ASSV2020,GJT2020,GTV2020,GTV2022a,GTV2022,AP2005, AKR2011, BGKT2008, BKT2009}. However, the nonlocal conservation laws \eqref{13456}, remain far from being settled, with the only exceptions being, \cite{ACG2015,ACT2015,BBKT2011, BG2016, GR2019, KPS2018} for smooth $F$, and \cite{CC2022,CV2021, KP2021} for partial results with spatially discontinuous $F(x,u)=\mathfrak{s}(x)u$. References \cite{ACG2015,ACT2015,BBKT2011,KPS2018} show the existence of entropy solutions via convergence of numerical approximations. Reference \cite{KP2021} establishes the same via a fixed-point argument with a bounded discontinuous $\mathfrak{s}$ and monotone non-symmetric kernel $\mu$. Further, the case of monotonically increasing piecewise constant $\mathfrak{s}$, with a single spatial discontinuity and $\mu$ being a downstream kernel, is handled in \cite{CC2022,CV2021}. The article \cite{CV2021} proves the existence via a convergent numerical scheme, using $\BV$ compactness away from the discontinuous interface, while \cite{CC2022} establishes the existence via a viscous approximation of \eqref{13456} and  proves the uniqueness with an appropriate interface condition. The question of well-posedness for \eqref{13456} with $F(x,u)$ being a non--linear function of $u$ and with spatial discontinuities (even single/finitely many) via any kinds of techniques, and for $F(x,u)=\mathfrak{s}(x)u$ with multiple finite/non discrete discontinuities via analysis of finite volume schemes,  remain unexplored and unsettled as of date.

The main difficulty in the analysis of numerical schemes is their compactness, which  in general, is counter intuitive and mathematically challenging even with ``local" discontinuous flux, {owing to the increase} in the total variation of the solution, (see for example, \cite{AJV2004,AMV2005,AKR2010,BGKT2008,GNPT2007,TOW2020}, and \cite{GJT2020,GTV2020,GTV2022a,GTV2022} for some recent developments). As a result, the compactness of the numerical approximations is non trivial and has a huge literature dedicated to its credit. The compactness, in general, is achieved via various techniques such as singular mapping, Flux TVD property and $\BV_{loc}$ bounds on the numerical approximations. Thus, it becomes interesting to study the convergence of numerical approximations of \eqref{13456} and extend the literature of techniques for local conservation laws to nonlocal ones. Since $F(x,u)$ is {nonlinear}, multiple weak solutions may exist, like for the local counterpart, and hence, an appropriate entropy condition is required, under which the solutions can be proved to be unique. 

This article settles these questions for \eqref{13456}

 in a {generic} setup, where  $F(\cdot,u)$ can admit {generic} spatial discontinuities (infinitely many with possible accumulation points), can possibly be a {nonlinear} function of $u$ and where \eqref{13456} takes the following form:
\begin{align}
 \label{eq:u}
  \partial_t u
  +
  \partial_x (f(\mathfrak{s}(x)u)\nu(\mu*\overline{\nu}(u)))
  &=
  0{\,\quad\quad \quad \quad \quad \text{for}\,\,\,(t,x) \in Q_{{T}}},
  \\ \label{eq:id}
  u(0,x)&=u_0(x){\, \, \quad \quad \quad \text{for}\,\,\,x \in \R},
\end{align} 
with

\begin{description}
\item[(H1)] $f \in \lip(\R)$  with $ f(0)=0,$ and $\mu,\nu,
\overline{\nu}\in {(C^2 \cap W^{2,\infty}   \cap \BV)}(\R)$ with ${\nu(0)=}\overline{\nu}(0)=0.$
\item[(H2)]  $\mathfrak{s}\in \BV(\R),$ with ${\inf_{x\in\R}\mathfrak{s}(x) >0}.$ 
\end{description}
Note that for the flux function $\overline{\nu}(u)=u,\mathfrak{s}=1,$ { and $\mu=\delta_0$ (Dirac delta distribution)} \eqref{eq:u} boils down to {the }generalized LWR model \eqref{LWR} which we began with. { The assumptions $f(0)=\nu(0)=\overline{\nu}(0)=0$ assures the physical bounds such as positivity and  stability estimates on the solution. Furthermore, these assumptions also imply that $u_0=0$ corresponds to {a }stationary state. Since $\mathfrak{s}$ is only $\BV,$ the flux function can be rough (depicting traffic on rough roads), i.e., $x\mapsto f(s(x)u)$ can be possibly discontinuous, potentially having an infinite number of points of discontinuity, including accumulation points.}
Moreover, the technical assumptions \textbf{(H1)}-\textbf{(H2)} are justified as these are the necessary technical assumptions required for well-posedness of the following scalar \textit{local} counterpart/nonlocal continuous version of \eqref{eq:u}-\eqref{eq:id},\begin{eqnarray}
 \label{eq:u11}
  \partial_t u
  +
  \partial_x (f(\mathfrak{s}(x)u) \nu(\overline{\nu}(u)))=0,
\end{eqnarray} see \cite{ACG2015,ACT2015,GTV2022a,GTV2022,PAN2009,TOW2020} and references therein. Since $f$ is {nonlinear}, an entropy condition is required to single out a unique solution, which we choose to be following:
\begin{definition} \label{def:entropy} A { function} $u\in C([0,T];\L1(\R))\cap {\L{\infty}([0,T];\BV(\R))}$ is an entropy solution of IVP~\eqref{eq:u}-\eqref{eq:id}, 

if for all $\alpha \in \R,$ {and for all non-negative $\phi\in C_c^{\infty}([0,T)\times \R)$}, the following holds:
\begin{eqnarray}\label{kruz}
&&\nonumber \int\limits_{Q_T} |u(t,x)-\mathfrak{s}_{\alpha}(x)|\phi_t\d{x} \d{t} \\\nonumber
&&+\int\limits_{Q_T}\sgn (u(t,x)-\mathfrak{s}_{\alpha}(x)) \mathcal{U}(t,x)(f(\mathfrak{s}(x)u)-f(\alpha))\phi_x \d{x} \d{t}\\\nonumber &&-\int\limits_{Q_T} f(\alpha) (\sgn (u(t,x)-\mathfrak{s}_{\alpha}(x))) \mathcal{U}_x(t,x)\phi \d{x} \d{t}\\
&&+\int\limits_{\R} |u_0(x)-\mathfrak{s}_{\alpha}(x)|\phi(0,x) \d x \geq 0,
\end{eqnarray}
where for $(t,x)\in Q_T,\, \mathcal{U}(t,x)= \nu(\mu*\overline{\nu}(u(t,x)))$ and $\mathfrak{s}_{\alpha}:\R \rightarrow \R$ is defined by {$\mathfrak{s}_{\alpha}(x):=\frac{\alpha}{\mathfrak{s}(x)}.$}
\end{definition}

In this article, a class of {numerical} schemes approximating \eqref{eq:u} will be introduced in \S\ref{main}. {This class} will be shown to have the required physical {properties such as}  positivity, mass conservation, $\L\infty$ bounds, along with the non-trivial global $\BV$ bound, which in general, {does} not hold for local conservation laws with discontinuous flux, making compactness arguments non-standard. Further, it will be shown in \S\ref{unique} that using the doubling of variables technique, any two entropy solutions of \eqref{eq:u} with the same initial data \eqref{eq:id}, satisfying \eqref{kruz} are {in fact} equal, and the approximate solutions obtained from the finite volume schemes satisfy the discrete form of entropy inequality \eqref{kruz}. As a result, the approximate solutions converge to the unique entropy solution of \eqref{eq:u}-\eqref{eq:id}. This uniqueness result is novel and is not dealt {with} in existing uniqueness results of \cite{BBKT2011,BG2016,KP2021}, which deal with either linear(discontinuous) or { nonlinear} smooth \textit{local} fluxes only. At the time of writing of this article, this is the first result establishing the existence and uniqueness of solutions of such nonlocal conservation laws with discontinuous flux, without any technical restrictions on monotonicity/finiteness of number of discontinuities of $\mathfrak{s}.$ {Hence, it }presents the first result in a { generic} setting, in contrast to the existing studies in the literature. 

The question on solutions of \eqref{eq:u}-\eqref{eq:id} converging to the solution of its local counterpart {\eqref{eq:u11}} as kernel support $\mu\rightarrow 0$, is far from fully understood and tackled. The current literature \cite{BS2020,BS2021,CCDKP2022,GNAL2021,CGES2021, 
CCS2019,FGKP2022,KL2022, KP2019} and the references therein, focus on  fluxes with only a \textit{linear} local part, where an additional entropy condition is not required to single out the unique solution\cite{CV2021,CAL2020, FKG2018, KP2021}.  Passage through the limit in \eqref{kruz} in case of { nonlinear} $f$, may not be straightforward. Though the article does not treat this question theoretically, numerical experiments  presented in \S\ref{NR} indicate that the convergence of nonlocal solutions to local solutions {might hold}. 

The paper is organized as follows. 
In \S\ref{unique}, we prove the uniqueness of the proposed entropy solution.
In \S\ref{main}, we present the convergence of { the} Godunov type and Lax-Friedrichs type schemes to the entropy solutions of the IVP \eqref{eq:u}-\eqref{eq:id}.  Finally, in \S\ref{NR}, we present some numerical experiments which illustrate the theory and also illustrate numerically, nonlocal to local behavior of such PDEs (see for example \cite{BS2021,CCS2019,CGES2021,CCMS2023, KP2022}).

\section{Uniqueness}\label{unique}
{In this section we prove the uniqueness of the entropy solutions, which will be used in \S\ref{NR} to show that the entire sequence of numerical approximations converge to the unique entropy solution of \eqref{eq:nl}-\eqref{eq:id}.}

For further use, for $(t,x)\in Q_T$, we define $\overline{u}(t,x):=\mathfrak{s}(x)u(t,x)$ so that $\sgn (u(t,x)-\mathfrak{s}_{\alpha}(x))=\sgn (\overline{u}(t,x)-\alpha)$, as $\mathfrak{s}>0$. Consequently, \eqref{kruz} {can be rewritten as:}
\begin{eqnarray}\label{kruz2}
&&\nonumber \int\limits_{Q_T}\left|u(t,x)-\frac{\alpha}{\mathfrak{s}(x)}\right|\phi_t\d{x} \d{t}\nonumber\\&&
+ \int\limits_{Q_T}\sgn (\overline{u}(t,x)-\alpha) (f(\overline{u}(t,x))-f(\alpha)) \mathcal{U}(t,x)\phi_x \d{x} \d{t}\nonumber\\\nonumber && -\int\limits_{Q_T} f(\alpha) (\sgn (\overline{u}(t,x)-\alpha)) \mathcal{U}_x(t,x)\phi \d{x} \d{t}\nonumber\\&&+\int\limits_{\R} |u_0(x)-\mathfrak{s}_{\alpha}(x)|\phi(0,x) \d x \geq 0 \quad  \quad \forall \, { 0\le } \phi\in C_c^{\infty}([0,T)\times \R).
\end{eqnarray}

\begin{theorem}\label{uniqueness}[Uniqueness of the entropy solutions]
Let $u,v$ be two entropy solutions of the IVP \eqref{eq:u}-\eqref{eq:id} corresponding to the initial data $u_0, v_0 {\,  \in (\L1\cap\BV)(\R)}$. Then, $u=v$ a.e. on { $\overline{Q}_T$.}
\end{theorem}
\begin{proof}
Let $u$ and $v$ be the entropy solutions to \eqref{eq:u}-\eqref{eq:id} corresponding to the initial data $u_0$ and $v_0$ respectively. Let $\mathcal{U}(t,x)= \nu(\mu*\overline{\nu}(u(t,x)))$  and $\mathcal{V}(s,y)= \nu(\mu*\overline{\nu}(v(s,y)))$. 
For $\phi=\phi(t,x,s,y) \in C_c^{\infty}(Q_T^2),$
 and for fixed $(s,y)\in Q_T$, we rewrite the entropy condition \eqref{kruz} for {$u$} with $\alpha=\overline{v}(s,y)$ as:
\begin{equation}\label{511}
\begin{array}{lll}&&0\leq\displaystyle \int\limits_{{Q_T}} \abs{u(t,x)-\frac{\overline{v}(s,y)}{\mathfrak{s}(x)}}\phi_t\d{x} \d{t}\\
&&+\displaystyle \int\limits_{{Q_T}}\sgn \left(\overline{u}(t,x)-\overline{v}(s,y)\right) \big(\left(f(\overline{u}(t,x))-f(\overline{v}(s,y))\right)\mathcal{U}(t,x)\phi_x\\&&\quad\qquad\quad\quad\quad\quad\quad\quad\quad\quad\quad\quad- f(\overline{v}(s,y))\mathcal{U}_x(t,x)\phi\big)\d{x} \d{t}.
\end{array}
\end{equation}
Using the relation
\begin{eqnarray*}
&&\left(f(\overline{u}(t,x))-f(\overline{v}(s,y))\right)\mathcal{U}(t,x)\phi_x-f(\overline{v}(s,y)) \mathcal{U}_x(t,x)\phi\\&&\quad \quad \quad \quad =\left(\left(\mathcal{V}(s,y)-\mathcal{U}(t,x)\right)f(\overline{v}(s,y))\phi\right)_x\\
&&\quad \quad \quad \quad\quad +\left(f(\overline{u}(t,x))\mathcal{U}(t,x)-f(\overline{v}(s,y))\mathcal{V}(s,y)\right)\phi_x,
\end{eqnarray*}
the inequality \eqref{511} can now be rewritten as:
\begin{align}\label{611}
&-\displaystyle \int\limits_{{Q_T}} \abs{u(t,x)-\frac{\overline{v}(s,y)}{\mathfrak{s}(x)}}\phi_t \d{x} \d{t}\\\nonumber
&-\displaystyle \int\limits_{{Q_T}}\sgn(\overline{u}(t,x)-\overline{v}(s,y))\left(f(\overline{u}(t,x))\mathcal{U}(t,x)-f(\overline{v}(s,y))\mathcal{V}(s,y)\right)\phi_x \d{x} \d{t} \\ \nonumber
&-\displaystyle \int\limits_{{Q_T}}\sgn\left( \overline{u}(t,x)-\overline{v}(s,y)\right)\left(\left(\mathcal{V}(s,y)-\mathcal{U}(t,x)\right)f(\overline{v}(s,y))\phi\right)_x\d{x} \d{t}
\leq 0.
\end{align}
Repeating the same arguments, { for a fixed $(t,x)\in Q_T,$ the entropy condition for $v$ }can be rewritten as:\begin{align}\label{6111}
&-\displaystyle \int\limits_{{Q_T}} \abs{v(s,y)-\frac{\overline{u}(t,x)}{\mathfrak{s}(y)}}\phi_s \d{y} \d{s}\\\nonumber
&-\displaystyle \int\limits_{{Q_T}}\sgn \left(\overline{v}(s,y)-\overline{u}(t,x)\right)\left(f(\overline{v}(s,y))\mathcal{V}(s,y)-f(\overline{u}(t,x))\mathcal{U}(t,x)\right)\phi_y \d{y} \d{s} \\\nonumber
&-\displaystyle \int\limits_{{Q_T}}\sgn\left( \overline{v}(s,y)-\overline{u}(t,x)\right)\left(\left(\mathcal{U}(t,x)-\mathcal{V}(s,y)\right)f(\overline{u}(t,x))\phi\right)_y\d{y} \d{s}
\leq 0.
\end{align}
Integrating \eqref{611} and \eqref{6111} in { $s,y$ and $t,x$ respectively} and adding, we get
\begin{eqnarray*}\label{31}
&&-\displaystyle \int\limits_{{{Q^2_T}}}\left( \abs{u(t,x)-\frac{\overline{v}(s,y)}{\mathfrak{s}(x)}}\phi_t+\abs{v(s,y)-\frac{\overline{u}(t,x)}{\mathfrak{s}(y)}}\phi_s \right) \d{x} \d{t} \d{y} \d{s}\\&&
-\displaystyle \int\limits_{{{Q^2_T}}}\sgn(\overline{u}(t,x)-\overline{v}(s,y))\left(f(\overline{u}(t,x))\mathcal{U}(t,x)-f(\overline{v}(s,y))\mathcal{V}(s,y)\right)(\phi_x+\phi_y) \d{x} \d{t} \d{y} \d{s}\\&&
-\displaystyle \int\limits_{{{Q^2_T}}} \sgn(\overline{u}(t,x)-\overline{v}(s,y)) \big(\left(\left(\mathcal{V}(s,y)-\mathcal{U}(t,x)\right)f(\overline{v}(s,y))\phi\right)_x\\&& \qquad \quad \quad \quad \quad \quad \quad  \quad \quad \quad\ \quad -\left(\left(\mathcal{U}(t,x)-\mathcal{V}(s,y)\right)f(\overline{u}(t,x))\phi\right)_y  \big)\d{x} \d{t} \d{y} \d{s}
\leq 0.
\end{eqnarray*}
{Rewriting} the above equations, we get:
\begin{eqnarray*}
- \displaystyle \int\limits_{{{Q^2_T}}}(I_0+I_1+I_2)\, \d{x} \d{t} \d{y} \d{s} \leq  0,
\end{eqnarray*}
\begin{align*}
I_0&=\abs{u(t,x)-\frac{\overline{v}(s,y)}{\mathfrak{s}(x)}}\phi_t+ \abs{v(s,y)-\frac{\overline{u}(t,x)}{\mathfrak{s}(y)}}\phi_s\\
I_1&=\sgn(\overline{u}(t,x)-\overline{v}(s,y))\left(f(\overline{u}(t,x))\mathcal{U}(t,x)-f(\overline{v}(s,y))\mathcal{V}(s,y))\right)(\phi_x+\phi_y)\\
I_2&= \sgn(\overline{u}(t,x)-\overline{v}(s,y)) \big(\left(\left(\mathcal{V}(s,y)-\mathcal{U}(t,x)\right)f(\overline{v}(s,y))\phi\right)_x \\&\quad\quad\quad\quad\quad\quad\quad\quad\quad\quad\quad-\left(\left(\mathcal{U}(t,x)-\mathcal{V}(s,y)\right)f(\overline{u}(t,x))\phi\right)_y\big).
\end{align*}
Next, we introduce a non-negative function $\xi\in \mathbf{\mathbf{C_c^{\infty}}}(\R)$ such that
    \[ \displaystyle \int_{\R}\xi(\sigma)\d{\sigma}=1, \xi(\sigma)=\xi(-\sigma),\xi(\sigma)=0, {\text{for all }}|\sigma|\ge 1,\] and set
    \[\xi_{\rho}(s):=\frac{1}{\rho}\xi\left(\frac{s}{\rho}\right),\quad  \rho\in \R^+,s\in \R.\]
   We then choose $\phi:=\Phi=\Phi(t,x,s,y)\in \mathbf{C_c^{\infty}}({Q^2_T}),$ by  \begin{eqnarray}\label{phi}\Phi(t,x,s,y)=\psi(t,x)\xi_{\rho}(t-s)\xi_{\rho}(x-y),
\end{eqnarray} where  $\psi=\psi(t,x)\in \mathbf{C_c^{\infty}}({Q_T})$ is a non-negative test function.
It is then straightforward to see that
\begin{align*}
&\left(\mathcal{V}(s,y)-\mathcal{U}(t,x)\right)f(\overline{v}(s,y))\Phi_x(t,x,s,y)\\&\qquad= \left(\mathcal{V}(s,y)-\mathcal{U}(t,x)\right)f(\overline{v}(s,y))\psi(t,x)\xi_{\rho}(t-s)\xi^{'}_{\rho}(x-y)\\
&\qquad\quad+\left(\mathcal{V}(s,y)-\mathcal{U}(t,x)\right)f(\overline{v}(s,y))\psi_x(t,x)\xi_{\rho}(t-s)\xi_{\rho}(x-y)
\\&\left(\mathcal{U}(t,x)-\mathcal{V}(s,y)\right)f(\overline{u}(t,x))\Phi_y(t,x,s,y)\\&\qquad = \left(\mathcal{V}(s,y)-\mathcal{U}(t,x)\right)f(\overline{u}(t,x))\psi(t,x)\xi_{\rho}(t-s)\xi^{'}_{\rho}(x-y).
\end{align*}
Now, { using the chain rule we can write}
\begin{align*}
I_2 &={\sgn(\overline{u}(t,x)-\overline{v}(s,y)) \big(\left(\mathcal{V}(s,y)-\mathcal{U}(t,x)\right)f(\overline{v}(s,y))\Phi_x }\\ &{\quad\quad\quad\quad\quad\quad\quad \quad\quad\quad\quad -\left(\mathcal{U}(t,x)-\mathcal{V}(s,y)\right)f(\overline{u}(t,x))\Phi_y\big)}\\
&\quad+{\color{black}\sgn(\overline{u}(t,x)-\overline{v}(s,y))\Phi \big(-\mathcal{U}_x(t,x)f(\overline{v}(s,y)) +\mathcal{V}_y(s,y)f(\overline{u}(t,x))\big)}\\
&:={\color{black}I_{2,1}}+{\color{black}I_{2,2}}.
\end{align*}
 We write $I_{2,1}=I_{2,1,1}+I_{2,1,2}$, where $I_{2,1,1}$ has { terms containing $\psi_x$ and $I_{2,1,2}$ has terms of $\psi,$}
i.e.,
\begin{align*}
I_{2,1,1}&=\sgn(\overline{u}(t,x)-\overline{v}(s,y))\left(\mathcal{V}(s,y)-\mathcal{U}(t,x)\right)f(\overline{v}(s,y))\psi_x(t,x)\xi_{\rho}(t-s)\xi_{\rho}(x-y)\\
I_{2,1,2}&=\sgn(\overline{u}(t,x)-\overline{v}(s,y))\left(\mathcal{V}(s,y)-\mathcal{U}(t,x)\right)(f(\overline{v}(s,y))-f(\overline{u}(t,x)))\\&\quad\quad\quad \times \psi(t,x)\xi_{\rho}(t-s) \xi^{'}_{\rho}(x-y)
\\
&=-\left(\mathcal{V}(s,y)-\mathcal{U}(t,x)\right)G(\overline{u}(t,x),\overline{v}(s,y))\psi(t,x)\xi_{\rho}(t-s)\xi^{'}_{\rho}(x-y),
\end{align*}
where \[G(a,b)=\sgn [a-b](f(a)-f(b)).\]
Applying integration by parts in $I_{2,1,2},$ 
we have \begin{align*}
&\displaystyle \int_{{Q^2_T}}I_{2,1,2}\d{t}\d{s}\d{x}\d{y}\\&=-
\displaystyle \int_{{Q^2_T}}\left(\mathcal{V}(s,y)-\mathcal{U}(t,x)\right)G(\overline{u}(t,x),\overline{v}(s,y))\psi(t,x)\xi_{\rho}(t-s)\xi^{'}_{\rho}(x-y)\d{t}\d{s}\d{x}\d{y}\\
&=\displaystyle \int_{{Q^2_T}}\left(-\mathcal{U}_x(t,x)G(\overline{u}(t,x),\overline{v}(s,y))+(\mathcal{V}(s,y)-\mathcal{U}(t,x))\partial_x G(\overline{u}(t,x),\overline{v}(s,y))\right)\\
&\quad\quad\quad\quad\times\psi(t,x)\xi_{\rho}(t-s)\xi_{\rho}(x-y)\d{t}\d{s}\d{x}\d{y}\\
&\quad+ \displaystyle \int_{{Q^2_T}}\left(\mathcal{V}(s,y)-\mathcal{U}(t,x)\right)G(\overline{u}(t,x),\overline{v}(s,y))\psi_x{(t,x)}\xi_{\rho}(t-s)\xi_{\rho}(x-y)\d{t}\d{s}\d{x}\d{y}.
\end{align*}
Consequently,
\begin{align*}
&\int_{Q_T^2}
(I_{2,1,2}+I_{2,2}) \d t \d x \d s \d y\\
&=\int_{Q_T^2} \big(\Phi \big(\sgn(\overline{u}(t,x)-\overline{v}(s,y))(\mathcal{V}_y(s,y)-\mathcal{U}_x(t,x))f(\overline{u}(t,x))\\&\quad\quad\quad\quad\quad+(\mathcal{V}(s,y)-\mathcal{U}(t,x))\partial_x G(\overline{u}(t,x),\overline{v}(s,y))\big)\\&\quad\quad\quad+\left(\mathcal{V}(s,y)-\mathcal{U}(t,x)\right)G(\overline{u}(t,x),\overline{v}(s,y))\psi_x{(t,x)}\xi_{\rho}(t-s)\xi_{\rho}(x-y) \big) \d t \d x \d s \d y.
\end{align*}
{Since $\mathfrak{s}\in \BV(\R)$ and $u\in L^{\infty}([0,T];\BV(\R)),$   we have $\overline {u} \in L^{\infty}([0,T];\BV(\R))$ and  consequently we have the following inequality in the sense of measures} (see \cite[Lem.~4.1]{KR2003} for details): \[\abs{\partial_x G(\overline{u}(t,x),\overline{v}(s,y))}\le {\abs{f}_{\lip(\R)}}\abs{\partial_x \overline{u}(t,x)},\]
{with  
\begin{align}
\nn \int_{\R}\abs{\partial_x \overline{u}(t,x)} \d x &\leq \abs{\overline{u}}_{L^{\infty}([0,T];\BV(\R))} \\ \nn & \leq \abs{\mathfrak{s}}_{\BV(\R)} \norma{u}_{L^{\infty}(Q_T)}+ \norma{\mathfrak{s}}_{L^{\infty}(\R)}\abs{u}_{L^{\infty}([0,T];\BV(\R))}\\&:=\mathcal{C}_1  <\infty. \label{C0}
\end{align}}
Sending the mollifier radius $\rho \downarrow 0$ in the above integrals, we get
\begin{eqnarray*}
&&\lim_{\rho\downarrow 0}\int_{{Q^2_T}}I_0 \d{t}\d{x}\d{s}\d{y}=\int_{{Q_T}}|u(t,x)-v(t,x)|\psi_t{(t,x)}\d{t}\d{x}\\&&
\lim_{\rho\downarrow 0}\int_{{Q^2_T}}I_1 \d{t}\d{x}\d{s}\d{y}=\int_{{Q_T}}\sgn [\overline{u}(t,x)-\overline{v}(t,x)]\left(g(\overline{u}(t,x))\mathcal{U}(t,x)-g(\overline{v}(t,x))\mathcal{V}(t,x)\right)\\&& \qquad \qquad \qquad \qquad \qquad \qquad\quad \times \psi_x{(t,x)}\d{t}\d{x}\\
&&\lim_{\rho\downarrow 0}\int_{{Q^2_T}}I_{2,1,1} \d{t}\d{x}\d{s}\d{y}=\int_{{Q_T}}\sgn [\overline{u}(t,x)-\overline{v}(t,x)]\left(\mathcal{V}(t,x)-\mathcal{U}(t,x)\right)g(\overline{v}(t,x))\\&& \quad \qquad \qquad \qquad \qquad \qquad\qquad \times\psi_x{(t,x)}\d{t}\d{x}\\&&
\lim_{\rho\downarrow 0}\int_{{Q^2_T}}(I_{2,1,2}+I_{2,2}) \d{t}\d{x}\d{s}\d{y} \le \int_{{Q_T}}\abs{\psi(t,x)} \big( {\abs{\mathcal{V}_x(t,x)-\mathcal{U}_x(t,x))}}{\abs{f}_{\lip(\R)}}\abs{ \overline{u}(t,x)}\\&&\quad \quad\quad \quad\quad \quad\quad \quad\quad \quad\quad \quad\quad\quad \quad \quad \quad\quad \quad +{\abs{\mathcal{V}(t,x)-\mathcal{U}(t,x)}}{\abs{f}_{\lip(\R)}}\abs{\partial_x \overline{u}(t,x)}\big)\d{t}\d{x}\\&&\quad \quad\quad \quad\quad \quad\quad \quad\quad \quad\quad \quad\quad\quad \quad+\int_{{Q_T}}\left|\mathcal{V}(t,x)-\mathcal{U}(t,x)\right| \abs{G(\overline{u}(t,x),\overline{v}(t,x))} \abs{\psi_x(t,x)}\d{t}\d{x}.
\end{eqnarray*} 
Now fix $0<t_1<t<t_2<T$ and choose
\begin{align*} \psi(t,x)&=\psi_{r,\theta}(t,x)=\Psi^1_{r}(x)\Psi^2_{{\theta}}(t)  &&  r>1, {\theta}>0,\\ 
\Psi^1_{r}(x)&=\int_{\R}\xi(|x-y|)\chi_{|y|<r}\d{y}  &&x\in \R, \\
\Psi^2_{{\theta}}(t)&=\int_{-\infty}^t
\left(\xi_{{\theta}}(\tau-t_1)-\xi_{{\theta}}(\tau-t_2)\right)d\tau    && 0<t_1<t<t_2<T,
\end{align*}
so that the terms containing $\psi_x$ will go to zero when $r\uparrow \infty,$ in effect implying that,
\begin{eqnarray*}
&&- \lim_{\rho \downarrow 0, r\uparrow \infty}\displaystyle \int\limits_{{Q^2_T} }(I_0+I_1+I_{2,1,1}+I_{2,1,2}+I_{2,2})\, \d{x} \d{t} \d{y} \d{s} \leq  0.
\end{eqnarray*}
Now, as $\theta\downarrow 0,$ we have 
\begin{align}
& \int_{\R}{\color{black}\abs{{u}(t_2,x)-{v}(t_2,x)}}\d{x}-\displaystyle \int_{\R}{\color{black}\abs{{u}(t_1,x)-{v}(t_1,x)}}\d{t}\d{x}\nonumber\\
&\leq  \displaystyle \int_{t_1}^{t_2}\displaystyle \int_{\R} \big({\norma{\mathcal{V}_x(t,\cdot)-\mathcal{U}_x(t,\cdot)}_{\L\infty({ \R})}}{\abs{f}_{\lip(\R)}}\abs{\overline{u}(t,x)} \nn \\& \quad \qquad \qquad +{\abs{f}_{\lip(\R)}}{\norma{\mathcal{V}(t,\cdot)-\mathcal{U}(t,\cdot)}_{\L\infty({\R})}}|\overline{u}
_x(t,x)|\big)\d{x}\d{t}\label{23}.
\end{align}
Further, { for $(t,x)\in Q_T,$ } the nonlocal weights $(\mathcal{U},\mathcal{V})$ satisfy,
\begin{align*}
\begin{split}
|\mathcal{V}(t,x)-\mathcal{U}(t,x)|&=|\nu((\mu*\overline{\nu}(u))(t,x))-\nu((\mu*\overline{\nu}(v))(t,x))|\\&\le{\norma{\nu'}_{\L\infty(\R)}}\abs{(\mu*(\overline{\nu}(u)-\overline{\nu}(v)))(t,x)}\\
&\le {\norma{\nu'}_{\L\infty(\R)}}\norma{\mu}_{\L\infty(\R)}{\norma{\overline{\nu}'}_{\L\infty(\R)}}\norma{u(t,.)-v(t,.)}_{\L1(\mathbb{R})}\\&{:=\mathcal{C}_2\norma{u(t,.)-v(t,.)}_{\L1(\mathbb{R})}},
\end{split}
\end{align*}
{which implies for $t\in [0,T],$
\begin{align}\label{UV}
\norma{\mathcal{V}(t,\cdot)-\mathcal{U}(t,\cdot)}_{L^{\infty}(\R)} \leq \mathcal{C}_2 \norma{u(t,.)-v(t,.)}_{\L1(\mathbb{R})}.
\end{align}}
Similarly, { for $(t,x)\in Q_T, $ }consider\begin{align*}
|&\mathcal{V}_x(t,x)-\mathcal{U}_x(t,x)|\\
&=|\nu'((\mu*\overline{\nu}(v))(t,x))(\mu'*\overline{\nu}(v))(t,x)-\nu'((\mu*\overline{\nu}(u))(t,x))(\mu'*\overline{\nu}(u))(t,x)|.
\end{align*}
Adding and subtracting $\nu'((\mu*\overline{\nu}(v))(t,x))(\mu'*\overline{\nu}(u))(t,x)$ we get:
\begin{align*}
\begin{split}|&\mathcal{V}_x(t,x)-\mathcal{U}_x(t,x)|
\\
&= \left|\nu'((\mu*\overline{\nu}(v))(t,x))(\mu'*\overline{\nu}(v))(t,x)-\nu'((\mu*\overline{\nu}(v))(t,x))(\mu'*\overline{\nu}(u))(t,x)\right|\\
&\quad +\left|\nu'((\mu*\overline{\nu}(v))(t,x))(\mu'*\overline{\nu}(u))(t,x)-\nu'((\mu*\overline{\nu}(u))(t,x))(\mu'*\overline{\nu}(u))(t,x)\right|\\
&= \left|\nu'((\mu*\overline{\nu}(v))(t,x))\left((\mu'*\overline{\nu}(v))(t,x)-(\mu'*\overline{\nu}(u))(t,x)\right)\right|\\
&\quad +\left|(\mu'*\overline{\nu}(u))(t,x)\right|\abs{\nu'((\mu*\overline{\nu}(v))(t,x))-\nu'((\mu*\overline{\nu}(u))(t,x))}\\&\le\norma{\nu'}_{\L\infty(\R)}{\norma{\overline{\nu}'}_{\L\infty(\R)}}\abs{(\mu'*(u-v)
)(t,x)}
\\&\quad+{\norma{\overline{\nu}'}_{\L\infty(\R)}}{\norma{\overline{\mu}'}_{\L\infty(\R)}}
{\norma{u}_{\L1(\R)}}{\norma{\nu''}_{\L\infty(\R)}}{\norma{\overline{\nu}'}_{\L\infty(\R)}}|(\mu*(u-v))(t,x)|\\
&\le \norma{\nu'}_{\L\infty(\R)}{\norma{\overline{\nu}'}_{\L\infty(\R)}}\norma{\mu'}_{\L\infty(\R)}\norma{u(t,.)-v(t,.)}_{\L1(\mathbb{R})}\\
&\quad+ {\norma{\overline{\nu}'}_{\L\infty(\R)}}{\norma{\overline{\mu}'}_{\L\infty(\R)}}
{\norma{u}_{\L1(\R)}}{\norma{\nu''}_{\L\infty(\R)}}{\norma{\overline{\nu}'}_{\L\infty(\R)}}\norma{\mu}_{\L\infty(\R)}\norma{u(t,.)-v(t,.)}_{\L1(\mathbb{R})}\\
&{:=(\mathcal{C}_3+\mathcal{C}_4)\norma{u(t,.)-v(t,.)}_{\L1(\mathbb{R})},}\end{split}\end{align*}
{which implies that for $t\in [0,T],$
\begin{align}\label{UV1}
\norma{\mathcal{V}_x(t,\cdot)-\mathcal{U}_x(t,\cdot)}_{L^{\infty}(\R)} \leq (\mathcal{C}_3+\mathcal{C}_4) \norma{u(t,\cdot)-v(t,\cdot)}_{L^1(\R)}.
\end{align}}
Finally, inserting the estimates \eqref{C0}, \eqref{UV} and \eqref{UV1} in \eqref{23}, we get:
\begin{eqnarray*}
&& \norma{{u}(t_2,\cdot)-{v}(t_2,\cdot)}_{\L1(\mathbb{R})}\le  \norma{{u}(t_1,\cdot)-{v}(t_1,\cdot)}_{\L1(\mathbb{R})}
+{\mathcal{C}_5} \displaystyle \int_{t_1}^{t_2}\norma{{u}(s,\cdot)-{v}(s,\cdot)}_{\L1(\mathbb{R})}\d{s},
\end{eqnarray*}
{ where $\mathcal{C}_5=\abs{f}_{\lip(\R)} \left(\mathcal{C}_1\mathcal{C}_2 +(\mathcal{C}_3+\mathcal{C}_4)\norma{\mathfrak{s}}_{L^{\infty}(\R)} \norma{u}_{L^{\infty}([0,T];L^1(\R))} \right).$}\\
Now, using Gronwall's inequality, {with $t_1=0$ and $t_2=t>0,$} we have \begin{eqnarray*}
{\norma{{u}(t,\cdot)-v(t,\cdot)}_{\L1(\mathbb{R})}\le \norma{u_0-v_0}_{\L1(\mathbb{R})}\exp(\mathcal{C}_5 t).}
\end{eqnarray*}
Now, taking $u_0=v_0,$ the theorem follows.
\end{proof}

\section{Numerical scheme and its convergence}\label{main}For $\Delta x, \Delta t>0$ { and  $\la:=\Delta t/\Delta x, $} we consider equidistant spatial grid points $x_i:=i\Delta x$ for $i\in\Z$ and temporal grid points $t^n:=n\Delta t$ 
for integers $0 \le n\le N$,  such that the final time $T \in [t^N, t^{N+1})$. Let $\chi_i(x)$ denote the indicator function of $C_i:=[x_i - \Delta x /2,  x_i + \Delta x /2)$,  and let
$\chi^n(t)$ denote the indicator function of $C^n:=[t^n, t^{n+1})$. Throughout, an initial datum ${u_0 \in (\L1}\cap \BV) (\R)$ is fixed and we approximate it according to:
{\begin{align*}
u^0_i=\int_{C_i}u_0(x)\d{x},  i\in \Z  \text{ and } u^{\D}_0(x):=\sum_{i\in \mathbb{Z}}\chi_{i}(x)u_i^0. \end{align*}}

Also, we approximate \begin{equation*}
\mathfrak{s}_i=\mathfrak{s}(x_i),  i\in \Z \text{ and } \mathfrak{s}^{\Delta}( x):= \sum_{i\in \mathbb{Z}}\chi_{i}(x)\mathfrak{s}_i.
\end{equation*}

{Further, 
the convolution term $\mu*\overline{\nu}(u)$ is approximated by 
\begin{eqnarray}
\label{eq:conv}
c_{\strut {i+1/2}}^{n}={\Delta x}\sum\limits_{p\in Z} {\mu(x_{{i+1/2}-p})}\overline{\nu}(u^{n}_{p+1/2}) \approx \int\limits_{\R} \mu(x-y)\overline{\nu}(u^{\D}(t,y))\d y, 
\end{eqnarray}
with $u^{n}_{p+1/2}$ being { any} convex combination of $u^{n}_{p}$ and $u^{n}_{p+1}.$ }
We now define a {piecewise constant approximate solution 
\begin{equation}
    u^{\Delta}(t, x):=\sum_{n=0}^N \sum_{i\in \mathbb{Z}}\chi_{i}(x)\chi^n(t) u^n_i,\quad
\end{equation}
 to the IVP \eqref{eq:u}{-\eqref{eq:id}
}through the following marching formula:}
\begin{align}\label{scheme2}.
\begin{split}
     u^{n+1}_i&=
   H(\nu(c^{n}_{{i-1/2}}),\nu(c^{n}_{{i+1/2}}),u_{i-1}^{n},u_i^{n},u_{i+1}^{n})
    \\&:=
   u^{n}_i- \lambda \bigl(
    \mathcal{F} (\nu(c^n_{{i+1/2}}),\mathfrak{s}_{i}u_i^{n},\mathfrak{s}_{i+1}u_{i+1}^{n})
    - 
    \mathcal{F} (\nu(c^n_{{i-1/2}}),\mathfrak{s}_{i-1}u_{i-1}^{n},\mathfrak{s}_{i}u_{i}^{n})
     \bigr).
\end{split}
  \end{align} 
{The numerical flux} $\mathcal{F} (\nu(c^n_{{i+1/2}}),\mathfrak{s}_{i}u_i^{n},\mathfrak{s}_{i+1}u_{i+1}^{n}) 
    $, denotes the numerical approximation of the flux $f(\mathfrak{s}(x)u) \nu(\mu*\overline{\nu}(u))$ at the interface $x=x_{{i+1/2}}$ for $i\in \Z.$ { Henceforth, we denote it through a shorthand notation  $\mathcal{F}^n_{{i+1/2}}(\mathfrak{s}_{i}u_i^{n},\mathfrak{s}_{i+1}u_{i+1}^{n})$.}
We present two examples of $\mathcal{F}$ that are the nonlocal extensions of well-known numerical monotone fluxes for local conservation laws. The other monotone fluxes can be modified in a similar way, as described below:
\begin{enumerate}
    \item \textbf{Lax-Friedrichs type Flux:} 
\begin{align*}
  \mathcal{F}_{LF}(a,b,c)
  & = 
  \displaystyle
  \frac{a}{2}\left( f(b)
    +
    f(c)\right)
  -
  \Theta\frac{(c-b)}{2\, \lambda} \,,\quad \quad \Theta\in\left(0,\displaystyle\frac{2}{3{\norma{\mathfrak{s}}_{\L\infty(\R)}}}\right), {\,a,b,c \in \R,}
\end{align*} where $\Delta t$ is chosen in order to satisfy
the CFL condition
\begin{equation}\label{CFl}
 \lambda \le \displaystyle\frac{\min(1, 4-6\Theta{\norma{\mathfrak{s}}_{\L\infty(\R)}},6\Theta {\norma{\mathfrak{s}}_{\L\infty(\R)}})}{1+6{\norma{\mathfrak{s}}_{\L\infty(\R)}}{\abs{f}_{\lip(\R)}}{\norma{\nu}_{\L\infty(\R)}}}.\end{equation}

\item \textbf{Godunov type flux:}
\begin{align*}
\mathcal{F}_{God}(a,b,c)&= a F_{God}(b,c), \quad { a,b,c \in \R,}\end{align*} 
where $F_{God}$ is the Godunov flux for the corresponding local conservation law $u_t+f(\mathfrak{s}(x)u)_x=0$, and  $\Delta t$ is chosen in order to satisfy
the CFL condition
\begin{align}\label{CFl1}
  \lambda {\norma{\mathfrak{s}}_{\L\infty(\R)}}{\abs{f}_{\lip(\R)}}{\norma{\nu}_{\L\infty(\R)}\le \frac{1}{6}}.
   \end{align}
\end{enumerate}
{
\begin{remark}\label{rem:mono}\normalfont
 For a fixed $n \in  \{0, \ldots, N\},$ the above choices of fluxes make $H$ increasing in the last three arguments, for an arbitrary fixed sequence $\{c_{{i+1/2}}^n\}_{i\in\Z}$. Henceforth, we refer to this property as the ``monotonicity" of the scheme (see \cite[Prop.~2.8]{ACT2015}, \cite[Thm.~5.1]{BGKT2008}). However, it is to be noted that while doing so, the dependence of the nonlocal coefficient on $c^n_{i+1/2}$ on $\{u^n_i\}_{i\in \Z}$ is suppressed. 
\end{remark}}
We present the analysis for  Lax-Friedrichs scheme. The analysis for { the }Godunov scheme can be done on similar lines. Let $\Delta_+z_{i}=z_{i+1}-z_{i}=\Delta_-z_{i+1},$ for any vector $z$. Now, we prove some lemmas on the solutions ${u^\Delta}$ which will lead to the {\color{black} convergence of the finite volume approximation ${u^{\Delta}}$ and hence }existence of weak entropy solution of \eqref{eq:u}.
 \begin{lemma}[Positivity]
  \label{lem:Pos}
Fix a ${u_0} \in (\L1  \cap \BV)
  (\R;{[0,\infty)})$. Then, the approximate solution ${u^\Delta}$
  defined by the {marching formula}~\eqref{scheme2} satisfies, for all , $t \in {\R^+}$ and $x\in \R,$
  \begin{displaymath}
    {u^{\Delta}} (t,x) \geq 0 \,.
  \end{displaymath}
\end{lemma}
\begin{proof}
{Let $i\in \Z$ and $n  \in  \{0, \ldots, N\}$.} Assume that $\mathfrak{s},u^{n}\ge 0.$ The scheme \eqref{scheme2} can be written { as:}
    \begin{equation*}
  \begin{array}{@{}l@{\,}c@{\,}l@{\,}c@{\,}c@{\,}c@{\,}c@{\,}c@{\,}c@{}}
u_{i}^{n+1}
     = 
 u_{i}^{n} &-& \lambda  \bigl[
   \mathcal{F}^{n}_{{i+1/2}} (\mathfrak{s}_iu_{i}^{n},\mathfrak{s}_{i+1}u_{i+1}^{n}){\color{black}-\mathcal{F}^{n}_{{i+1/2}} (\mathfrak{s}_iu_{i}^{n},\mathfrak{s}_iu_{i}^{n})}\\[2mm]
    & &{\color{black}+\mathcal{F}^{n}_{{i+1/2}} (\mathfrak{s}_iu_{i}^{n},\mathfrak{s}_iu_{i}^{n})}{\color{black}- \mathcal{F}^{n}_{{i-1/2}} (\mathfrak{s}_iu_{i}^{n},\mathfrak{s}_iu_{i}^{n})}\\[2mm]
    & &{\color{black}+\mathcal{F}^{n}_{{i-1/2}} (\mathfrak{s}_iu_{i}^{n},\mathfrak{s}_iu_{i}^{n})}-\mathcal{F}^{n}_{{i-1/2}} (\mathfrak{s}_{i-1}u_{i-1}^{n},\mathfrak{s}_iu_{i}^{n}) \bigr],
  \end{array}
\end{equation*}
which { can be written in the following incremental form}:
 \begin{equation} \label{scheme21}
    \begin{array}{rcr}
     u_{i}^{n+1}
=
      u_{i}^{n}
      -a_{i-1/2}^n \,  \Delta_-(\mathfrak{s}_iu_i^{n})
      + b_{i+1/2}^n \,  \Delta_+(\mathfrak{s}_iu_i^{n}
      )-
      \lambda \left(
        {\color{black} \mathcal{F}^{n}_{{i+1/2}} (\mathfrak{s}_iu_{i}^{n},\mathfrak{s}_iu_{i}^{n})}
        -
          {\color{black}\mathcal{F}^{n}_{{i-1/2}} (\mathfrak{s}_iu_{i}^{n},\mathfrak{s}_iu_{i}^{n})}
      \right),
    \end{array}
  \end{equation}
  where
  \begin{align*}
   a_{i-1/2}^n
    & = 
    \lambda \,
    \frac{
       {\color{black}\mathcal{F}^{n}_{{i-1/2}} (\mathfrak{s}_iu_{i}^{n},\mathfrak{s}_iu_{i}^{n})}
      -
     \mathcal{F}^{n}_{{i-1/2}} (\mathfrak{s}_{i-1}u_{i-1}^{n},\mathfrak{s}_iu_{i}^{n})}{\Delta_-(\mathfrak{s}_iu_i^{n})} \,,\\
    b_{i+1/2}^n
    & = 
    \lambda \,
    \frac{
       {\color{black}\mathcal{F}^{n}_{{i+1/2}} (\mathfrak{s}_iu_{i}^{n},\mathfrak{s}_{i+1}u_{i+1}^{n})}
      -
      \mathcal{F}^{n}_{{i+1/2}} (\mathfrak{s}_iu_{i}^{n},\mathfrak{s}_iu_{i}^{n})}{\Delta_+(\mathfrak{s}_iu_i^{n})}\,.
  \end{align*}
Indeed,
  \begin{align*}
a^{n}_{{i-1/2}}
    & = 
    \lambda\frac{\nu(c^{n}_{{i-1/2}})}{\Delta_-(\mathfrak{s}_iu_i^{n})}
    \left(
      f(\mathfrak{s}_iu_{i}^{n})
      -\frac{
        f(\mathfrak{s}_{i-1}u_{i-1}^{n})
        +
       f(\mathfrak{s}_iu_{i}^{n})
      }{2}
    \right)
    +\frac{\Theta}{2}
    \\
     & = 
    \lambda \frac{ \nu({c_{{i-1/2}}^{n}})}{{\Delta_-{(\mathfrak{s}_{i}u_i^n)}}}
   \left(
    \frac{\Delta_-{(f(\mathfrak{s}_{i}u_i^n))}}{2}\right)+\displaystyle\frac{\Theta}{2}
    \\
    & = 
   \frac{1}{2}\nu({c_{{i-1/2}}^{n}})
   \left(
    \lambda f' (\xi_{{{i-1/2}}}^{n})\right)+\displaystyle\frac{\Theta}{2},
  \end{align*}where { the last equality follows from the mean value theorem with }$\xi_{{i-1/2}}^{n}\in I(\mathfrak{s}_{i-1}u_{i-1}^n,\mathfrak{s}_iu_i^n).$\\ 
  Hence, using the CFL condition
   {cf.}~\eqref{CFl},
  \begin{align}
  \begin{split}\label{eq:bounda}
    \mathfrak{s}_{i}a^{n}_{{i-1/2}}
    &\geq
    \norma{\mathfrak{s}}_{\L\infty(\R)}\frac{-\lambda \norma{\nu}_{\L\infty(\R)}\abs{f}_{\lip(\R)} + \Theta}{2}
    \geq
    0,
    \\
\mathfrak{s}_{i}a^{n}_{{i-1/2}}
    &\leq
\norma{\mathfrak{s}}_{\L\infty(\R)}\frac{\lambda \norma{\nu}_{\L\infty(\R)}\abs{f}_{\lip(\R)} + \Theta}{2}
    \leq
    \frac{1}{3}. 
    \end{split}
  \end{align}
 Entirely similar computations hold for the term
$b^{n}_{{i+1/2}}$ { and hence
$\mathfrak{s}_{i}a_{i+1/2}^n$ and $\mathfrak{s}_{i}b_{i+1/2}^n$ are non negative with $1-\mathfrak{s}_{i}a_{i-1/2}^n-\mathfrak{s}_{i}
   b_{i+1/2}^n \geq 0.$}
  Further, since $\mathfrak{s}_{i}u_i^n>0,$ and {$f(0)=0,$}
  \begin{align*}
    &
    \modulo{ \mathcal{F}^{n}_{{i+1/2}} (\mathfrak{s}_iu_{i}^{n},\mathfrak{s}_iu_{i}^{n})
      -
      \mathcal{F}^{n}_{{i-1/2}} (\mathfrak{s}_iu_{i}^{n},\mathfrak{s}_iu_{i}^{n})}
    \\
    & = 
f(\mathfrak{s}_iu_i^n)\left|
     \nu(c_{i+1/2}^n)
      -
      \nu(c_{i-1/2}^n)
\right| \leq 
    2\abs{f}_{\lip(\R)}\mathfrak{s}_iu_i^n\norma{\nu}_{\L\infty(\R)},
  \end{align*}    
which under the choice of $\Delta t,$  according to \eqref{CFl}, implies:
  \begin{equation}    \modulo{\mathcal{F}^{n}_{{i+1/2}} (\mathfrak{s}_iu_{i}^{n},\mathfrak{s}_iu_{i}^{n})
      -
      \mathcal{F}^{n}_{{i-1/2}} (\mathfrak{s}_iu_{i}^{n},\mathfrak{s}_iu_{i}^{n})}\leq
    \frac{1}{3\lambda}u_i^n. \label{eq:boundF}
  \end{equation}
  Using~\eqref{eq:bounda}-\eqref{eq:boundF} in~\eqref{scheme2} we
  get:
  \begin{align}
  \label{2.9}
  \begin{split}
  u_{i}^{n+1}
    & = 
    (1-\mathfrak{s}_{i}a^{n}_{{i-1/2}} -\mathfrak{s}_{i} b^{n}_{{i+1/2}}) u_{i}^{n}
    +
 a^{n}_{{i-1/2}} \mathfrak{s}_{i-1}u_{i-1}^{n}
    +
 b^{n}_{{i+1/2}} \mathfrak{s}_{i+1}u_{i+1}^{n}
    \\
    & 
    \quad -
    \lambda   \left(
      \mathcal{F}^{n}_{{i+1/2}} (\mathfrak{s}_iu_{i}^{n},\mathfrak{s}_iu_{i}^{n})
      -
      \mathcal{F}^{n}_{{i-1/2}} (\mathfrak{s}_iu_{i}^{n},\mathfrak{s}_iu_{i}^{n})
    \right) \end{split}
    \\ \nonumber
    & \geq 
    \left(\frac{2}{3}-\mathfrak{s}_{i}a^{n}_{{i-1/2}} - \mathfrak{s}_{i}b^{n}_{{i+1/2}}   \right) u_{i}^{n}
    +
    a^{n}_{{i-1/2}} \mathfrak{s}_{i-1}u_{i-1}^{n}
    +
    b^{n}_{{i+1/2}} \mathfrak{s}_{i+1}u_{i+1}^{n}
     \geq 
    0,
    \end{align}
  {which implies the desired result by employing induction.} 
\end{proof}
The following lemma is an easy consequence of positivity of ${u^{\Delta}}$.
\begin{lemma}[$\L1$ bound]
  \label{lem:L1}
  Let \eqref{CFl}
  hold. Fix an initial datum ${u_0} \in (\L1  \cap \BV)
  (\R;{[0,\infty)})$. Then, the approximate solution ${u^\Delta}$
  defined by the {marching formula}~\eqref{scheme2} satisfies, for $t \in {\R^+}$,
  \begin{displaymath}
    {\norma{ {u^{\Delta}} (t)}_{\L1(\R)}} = {\norma{ {u^{\Delta}} (0)}_{\L1(\R)}} \,.
  \end{displaymath}
\end{lemma}
\begin{proof}
{The proof follows by a standard computation, see for example \cite[Lem.~2.4]{ACT2015}.}
\end{proof}
To establish the $\L\infty$ bounds, we need the following lemma on the convolution terms.
 \begin{lemma}
  \label{lem:AB}
The convolution terms { \eqref{eq:conv} }satisfy
  \begin{align}
\label{eq:c_i+1/2}\begin{split}
    |c^n_{{i+1/2}} - c^n_{{i-1/2}}|
    &\leq 
    \mathcal{K}_1\Delta x, \,\\
    |c^n_{{i+3/2}} -2c^n_{{i+1/2}}+ c^n_{{i-1/2}}|
    &\leq 
 \mathcal{K}_2\, \Delta x^{2},
  \end{split}
    \end{align}
    where $$\mathcal{K}_1={\norma{\overline{\nu}'}_{\L\infty(\R)}}{{\norma{\mu'}_{\L\infty(\R)}} \norma{ {u^{\Delta}}(t^n)}_{\L1(\R)}}  \text{ and }\mathcal{K}_2=2{\norma{\overline{\nu}'}_{\L\infty(\R)}}{\norma{\mu^{''}}_{\L\infty(\R)}}
  {\norma{ {u^{\Delta}}(t^n)}_{\L1(\R)}}.$$
  \end{lemma}
  \begin{proof}
  {Let $i\in \Z$ and $n  \in  \{0, \ldots, N\}$}. Consider,
 \begin{align*}
    & 
   \abs{c^{n}_{{i+1/2}} - c^{n}_{{i-1/2}}}
    \\
     & \leq 
     \Delta x \sum_{l\in \interi}
    \abs{\mu(x_{{i+1/2-l}}) - \mu(x_{{i-1/2}-l})}\abs{ \,
     \overline{\nu}(u^n _{l+1/2})}\\
     & \leq 
    \Delta x \sum_{l\in \interi}\abs{
     \overline{\nu}(u^n _{l+1/2})}\int_{x_{{i-1/2}-l}}^{x_{{i+1/2-l}}}\abs{\mu'(s)}\d s\\
    & \leq \,  \Delta x{\norma{\overline{\nu}'}_{\L\infty(\R)}}{\norma{\mu'}_{\L\infty(\R)}}{\norma{ {u^{\Delta}}(t^n)}_{\L1(\R)}} 
    \end{align*}
    \begin{align*}
    & 
   \abs{c^{n}_{{i+3/2}} - 2c^{n}_{{i+1/2}}+c^{n}_{{i+1/2}}}
    \\
     & \leq 
     \Delta x \sum_{l\in \interi}
    \abs{\mu(x_{{i+1/2-l}}) - \mu(x_{{i-1/2}-l})}\abs{ \,
     \overline{\nu}(u^n _{l+1/2})}\\
     & \leq 
    \Delta x \sum_{l\in \interi}\abs{
     \overline{\nu}(u^n _{l+1/2})}\int_{x_{{i-1/2}-l}}^{x_{{i+1/2-l}}}\abs{\mu'(s)}\d s\\
    & \leq \,  \Delta x{\norma{\overline{\nu}'}_{\L\infty(\R)}}{\norma{\mu'}_{\L\infty(\R)}}{\norma{ {u^{\Delta}}(t^n)}_{\L1(\R)}}   
    \end{align*}
    The remaining inequality also follows similarly, see \cite[Proposition~2.8]{ACT2015}.
  \end{proof}
We now establish the $\L\infty$ bounds.
  \begin{lemma}[$\L\infty$ bound]
  \label{lem:Linfty}
  Let~~\eqref{CFl}
  hold. Fix an initial datum ${u_0} \in (\L1  \cap \BV)
  (\R;{[0,\infty)})$. Then, the approximate solution ${u^\Delta}$
  defined by the {marching formula}~\eqref{scheme2} satisfies for all $t \in
  {\R^+}$,
  \begin{equation}
    \label{eq:8}
    \norma{ {u^{\Delta}} (t)}_{\L\infty(\R)}
    \leq
    \mathcal{K}_4\exp(\mathcal{K}_3 \, t ) \,
    \norma{{u^{\Delta}(0)}}_{\L\infty(\R)} \, ,
  \end{equation}
  with {\begin{equation*}
 \mathcal{K}_3
    =
    \mathcal{K}_1\Delta t\abs{f}_{\lip(\R)}\norma{\mathfrak{s}}_{\L\infty(\R)}
     { \norma{\nu'}_{\L\infty(\R)}},  \,\, \mathcal{K}_4=\frac{\norma{\mathfrak{s}}_{\L\infty(\R)}}{{ \inf\limits_{x\in \R}\mathfrak{s}(x)}}. 
  \end{equation*}}
\end{lemma}
\begin{proof}
{Let $i\in \Z$ and $n  \in  \{0, \ldots, N\}$. Using Lemma~\ref{lem:AB}}, we obtain:
  \begin{align}
    & 
    \mathcal{F}^{n}_{{i+1/2}} (\mathfrak{s}_iu_i^{n},\mathfrak{s}_iu_i^{n})
    -
    \mathcal{F}^{n}_{{i-1/2}} (\mathfrak{s}_iu_i^{n},\mathfrak{s}_iu_i^{n})
    \nonumber
    \\
    & = 
  f(\mathfrak{s}_iu_i^{n})\left(\nu(c_{{i+1/2}}^n)
    -
    \nu(c_{{i-1/2}}^n)\right)
    \nonumber
    \\
    & \leq 
   \abs{f}_{\lip(\R)}\norma{\mathfrak{s}}_{\L\infty(\R)}u_{i}^{n}
     { \norma{\nu'}_{\L\infty(\R)}}     \mathcal{K}_1 \Delta x .
    \label{eq:bound1}
  \end{align}
  Now, using Lemma \ref{lem:L1}, \eqref{eq:bound1} in~\eqref{2.9}, we get
  \begin{align*}
    \modulo{ \mathfrak{s}_iu_{i}^{n+1}}
    & \leq 
    (1-\mathfrak{s}_{i}a^{n}_{{i-1/2}}-\mathfrak{s}_{i}b^{n}_{{i+1/2}} ) \mathfrak{s}_iu_i^{n}
    +
    \mathfrak{s}_i a^{n}_{{i-1/2}} \mathfrak{s}_{i-1}u_{i-1}^{n}
    \\
    & 
    \quad
    +
     \mathfrak{s}_ib^{n}_{{i+1/2}} \mathfrak{s}_{i+1}u_{i+1}^{n}
    +
     \mathfrak{s}_i\abs{f}_{\lip(\R)}\norma{\mathfrak{s}}_{\L\infty(\R)}u_{i}^{n}
     { \norma{\nu'}_{\L\infty(\R)}}     \mathcal{K}_1 \Delta t
    \\
    & \leq 
    (1-\mathfrak{s}_{i}a^{n}_{{i-1/2}}-\mathfrak{s}_{i}b^{n}_{{i+1/2}}) \norma{ \mathfrak{s}{u^{\Delta}}(t^n)}_{\L\infty(\R)}
    +
     \mathfrak{s}_ia^{n}_{{i-1/2}} \norma{ \mathfrak{s}{u^{\Delta}}(t^n)}_{\L\infty(\R)}
    \\
    & 
    \quad
    +
    \mathfrak{s}_{i}b^{n}_{{i+1/2}} {\norma{\mathfrak{s}{u^{\Delta}}(t^n)}_{\L\infty(\R)}}
    + {\abs{f}_{\lip(\R)}}\norma{\mathfrak{s}}_{\L\infty(\R)}
     { \norma{\nu'}_{\L\infty(\R)}}     \mathcal{K}_1 \Delta t{\norma{\mathfrak{s}{u^{\Delta}}(t^n)}_{\L\infty(\R)}}
    \\
    & \leq 
    {\norma{\mathfrak{s}{u^{\Delta}}(t^n)}_{\L\infty(\R)}}
    \left(
      1
      +
      \mathcal{K}_1\Delta t\abs{f}_{\lip(\R)}\norma{\mathfrak{s}}_{\L\infty(\R)}
    { \norma{\nu'}_{\L\infty(\R)}} 
    \right)\,.
  \end{align*}
{Consequently, with $\mathcal{K}_3=\mathcal{K}_1\abs{f}_{\lip(\R)}\norma{\mathfrak{s}}_{\L\infty(\R)}
    { \norma{\nu'}_{\L\infty(\R)}}$, we have}
  \begin{align*}
    \norma{\mathfrak{s}{u^{\Delta}}(t^{n+1})}_{\L\infty(\R)}
   &\leq
    \left(
      1
      +
    \mathcal{K}_3 \Delta t
    \right) {\norma{\mathfrak{s}{u^{\Delta}}(t^n)}_{\L\infty(\R)}} \,
\\&\le \exp(\mathcal{K}_3\Delta t){\norma{\mathfrak{s}{u^{\Delta}}(t^n)}_{\L\infty(\R)}}\\
&\le \exp(\mathcal{K}_3(n+1)\Delta t){\norma{\mathfrak{s}}_{\L\infty(\R)}}\norma{{u^{\Delta}(0)}}_{\L\infty(\R)},
 \end{align*}
{\color{black} so that for $0 \leq n \leq N,$} 
  \begin{displaymath}  \norma{{u^{\Delta}}(t^n)}_{\L\infty(\R)}
    \leq
    \frac{\norma{\mathfrak{s}}_{\L\infty(\R)}}{{\inf\limits_{x\in \R}\mathfrak{s}(x)}}{\norma{{u^{\Delta}(0)}}_{\L\infty(\R)}}
    \exp(\mathcal{K}_3t^n)
  \end{displaymath}
  completing the proof.
  \end{proof}

  {
  
      In general, the maximum principle does not hold for nonlocal conservation laws with discontinuous flux, i.e., $\norma{u(t,)}_{L^{\infty}(\R)} \nleq \norma{u_0}_{L^{\infty}(\R)}.$ However, if the flux function in addition satisfies $f(0)=f(1)=0$ and $ \mathfrak{s} \geq 1,$ then the entropy solutions indeed exhibit the `invariant region' principle. More precisely we have the following lemma.
  \begin{lemma}[Invariant region principle]\label{lem:invariant}
      In addition to the hypothesis of the Lemma~\ref{lem:Linfty}, assume that the flux function satisfies $f(0)=f(1)=0$ and $\mathfrak{s}\geq 1,$ then $0\leq u_0 \leq 1$ implies $0 \leq u(t,\cdot) \leq 1,$ for $t\geq 0.$
  \end{lemma}
  \begin{proof}
  In view of  Lemma \ref{lem:Pos}, enough to show that $u(t,\dott) \leq 1$ for all $t\geq 0.$
  From the marching formula \eqref{scheme2} we have
  \begin{align}.
\begin{split}
     \mathfrak{s}_iu^{n+1}_i&=
   \mathfrak{s}_iu^{n}_i- \lambda \mathfrak{s}_i\bigl(
    \mathcal{F} (\nu(c^n_{{i+1/2}}),\mathfrak{s}_{i}u_i^{n},\mathfrak{s}_{i+1}u_{i+1}^{n})
    - 
    \mathcal{F} (\nu(c^n_{{i-1/2}}),\mathfrak{s}_{i-1}u_{i-1}^{n},\mathfrak{s}_{i}u_{i}^{n})
     \bigr)\\
&:=
     \tilde{H}(\nu(c^{n}_{{i-1/2}}),\nu(c^{n}_{{i+1/2}}),\mathfrak{s}_{i-1}u_{i-1}^{n},\mathfrak{s}_{i}u_i^{n},\mathfrak{s}_{i+1}u_{i+1}^{n}).
\end{split}
  \end{align} 
Note that  $\tilde{H}(a,b,1,1,1)=1,$ for all $a,b \in \R.$  
Furthermore, since the marching operator $H$ is monotone (see \eqref{scheme2} and Remark~\ref{rem:mono}), the operator $\tilde{H}$ is increasing in its last three arguments. Consequently, by inductive argument we have:
\begin{align*}
\mathfrak{s}_iu_i^{n+1}&:= \tilde{H}(\nu(c^{n}_{{i-1/2}}),\nu(c^{n}_{{i+1/2}}),\mathfrak{s}_{i-1}u_{i-1}^{n},\mathfrak{s}_{i}u_i^{n},\mathfrak{s}_{i+1}u_{i+1}^{n}) \leq \tilde{H}(\nu(c^{n}_{{i-1/2}}),\nu(c^{n}_{{i+1/2}}),1,1,1)=1,
\end{align*}
which implies the lemma.
  \end{proof}
  }
    {Now, we establish uniform $\BV$ bounds on the solutions $u^{\Delta},$ which in general, do not hold for local conservation laws with discontinuous flux (see \cite{AGDV2011, GHO2015,GTV2020} and references therein for details). The special form of the flux i.e., $F(x,u)=f(\mathfrak{s}(x)u)$ enables us to obtain such BV bounds  even though the spatial discontinuities can be infinitely many with possible accumulation points.}    
\begin{lemma}\label{lem:BV}($\BV$ estimate):
Let ~\eqref{CFl} hold. Fix an initial datum ${u_0} \in (\L1  \cap \BV)
  (\R;{[0,\infty)})$. Then, the approximate solution
  ${u^\Delta}$ defined by the {marching formula}~\eqref{scheme2} satisfies for all
  $n$, for all $t \in \left(t^n, t^{n+1}\right),$ 
    \begin{align}    
    \begin{split}
    \label{BV:u}
    \sum_{i\in\mathbb{Z}}
        \modulo{\Delta_+{u_i^{n}}} 
          &\leq \frac{1}{\inf\limits_{x\in \R}\mathfrak{s}(x)}\bigg(
     \exp(\mathcal{K}_5t)\sum_{i\in\Z}\abs{\Delta_+{(\mathfrak{s}_{i}u_{i}^{0})}}  + \frac{\exp(\mathcal{K}_5t)-1}{\mathcal{K}_5}\mathcal{K}_6
      \\& \qquad \qquad \qquad + \mathcal{K}_4\exp(\mathcal{K}_3 \, t ) \abs{\mathfrak{s}}_{\BV(\R)} \norma{u^{\Delta}(0)}_{L^{\infty}(\R)}\bigg),
     \end{split}
    \end{align}
    where \begin{align*}
\mathcal{K}_5&=\mathcal{K}_1{\norma{\mathfrak{s}}}_{\L\infty({\mathbb{R}})} \abs{f}_{\lip(\R)}  { \norma{\nu'}_{\L\infty(\R)}},\\
\mathcal{K}_6&=\mathcal{K}_1{\abs{\mathfrak{s}}_{\BV(\R)}}
\norma{\mathfrak{s}}_{\L\infty({\mathbb{R}})}\abs{f}_{\lip(\R)}  { \norma{\nu'}_{\L\infty(\R)}}\norma{u^{n}}_{\L\infty}+\mathcal{K}_2{\norma{\mathfrak{s}}}_{\L\infty({\mathbb{R}})}\abs{f}_{\lip(\R)}{\norma{\nu'}_{\L\infty(\R)}} \norma{\mathfrak{s}u^{n}}_{\L1}\\
&\quad{+2\mathcal{K}^2_1\norma{\mathfrak{s}}_{\L\infty(\R)}\abs{f}_{\lip(\R)}\abs{\nu'}_{\lipR}\norma{\mathfrak{s}u^{n}}_{\L1}}.
      \end{align*}
    \end{lemma}
    \begin{proof}
    {Let $i\in \Z$ and $n  \in  \{0, \ldots, N\}$.} Using { the incremental form} \eqref{scheme21}, we have, 
   \begin{align*}
     \mathfrak{s}_iu_{i}^{n+1}&= 
      \mathfrak{s}_iu_i^{n}
      -
       \mathfrak{s}_ia_{i-1/2}^n \,  \Delta_-(\mathfrak{s}_iu_i^{n})
      +
      \mathfrak{s}_i b_{i+1/2}^n \,  \Delta_+(\mathfrak{s}_iu_i^{n})
\\&\quad \quad \quad-
      \lambda  \mathfrak{s}_i\left(
        {\color{black}\mathcal{F}^{n}_{{i+1/2}} (\mathfrak{s}_iu_{i}^{n},\mathfrak{s}_iu_{i}^{n})}
        -
          {\color{black}\mathcal{F}^{n}_{{i-1/2}} (\mathfrak{s}_iu_{i}^{n},\mathfrak{s}_iu_{i}^{n})}
      \right) \,
  \\
 \mathfrak{s}_{i+1}u^{n+1}_{i+1}&= 
     \mathfrak{s}_{i+1}u^{n+1}_{i+1}
      -\mathfrak{s}_{i+1}
      a_{i+1/2}^n \, \Delta_+(\mathfrak{s}_iu_i^{n})
      +  \mathfrak{s}_{i+1}
      b_{i+3/2}^n \, \Delta_+ {(\mathfrak{s}_{i+1}u_{i+1}^{n})}
   \\ &\quad \quad \quad -
      \lambda   \mathfrak{s}_{i+1}\left(
        {\color{black}\mathcal{F}^{n}_{{i+3/2}} (\mathfrak{s}_{i+1}u_{i+1}^{n},\mathfrak{s}_{i+1}u_{i+1}^{n})}
        -
          {\color{black}\mathcal{F}^{n}_{{i+1/2}}(\mathfrak{s}_{i+1}u_{i+1}^{n},\mathfrak{s}_{i+1}u_{i+1}^{n})}
      \right) \,.
     \end{align*}
Subtracting the above equations we get
   \begin{align}
     &\Delta_+(\mathfrak{s}_iu_i^{n+1})\nonumber
     \\\nonumber
     & = \Delta_+(\mathfrak{s}_iu_i^{n})  -\mathfrak{s}_{i+1}
      a_{i+1/2}^n \, \Delta_+(\mathfrak{s}_iu_i^{n})+
        \mathfrak{s}_ia_{i-1/2}^n \, \Delta_-(\mathfrak{s}_iu_i^{n})
      +  \mathfrak{s}_{i+1}
      b_{i+3/2}^n \, \Delta_+ {(\mathfrak{s}_{i+1}u_{i+1}^{n})}\\
      &\nonumber
      \quad- \mathfrak{s}_ib_{i+1/2}^n \, \Delta_+(\mathfrak{s}_iu_i^{n})
        -\mathfrak{s}_{i+1}\lambda \left(
        \mathcal{F}^{n}_{{i+3/2}} (\mathfrak{s}_{i+1}u_{i+1}^{n},\mathfrak{s}_{i+1}u_{i+1}^{n})
        -
          {\color{black}\mathcal{F}^{n}_{{i+1/2}}(\mathfrak{s}_{i+1}u_{i+1}^{n},\mathfrak{s}_{i+1}u_{i+1}^{n})}
      \right) \\&\nonumber \quad +\lambda  \mathfrak{s}_i \left(
        {\color{black}\mathcal{F}^{n}_{{i+1/2}} (\mathfrak{s}_iu_{i}^{n},\mathfrak{s}_iu_{i}^{n})}
        -
          {\color{black}\mathcal{F}^{n}_{{i-1/2}} (\mathfrak{s}_iu_{i}^{n},\mathfrak{s}_iu_{i}^{n})}
      \right) \\ \label{ABC}&:=A_i+B_i+C_i,\end{align}
      where
      \begin{align}
      A_i&=\Delta_+(\mathfrak{s}_iu_i^{n}) \left(1-\mathfrak{s}_{i+1}
      a_{i+1/2}^n-\mathfrak{s}_{i}
      b_{i+1/2}^n\right)+\mathfrak{s}_{i}
      a_{i-1/2}^n\Delta_-(\mathfrak{s}_iu_i^{n}) +\mathfrak{s}_{i+1}
      b_{i+3/2}^n\Delta_+(\mathfrak{s}_{i+1}u_{i+1}^{n})\label{1} \\
       \label{2}B_i&=-\mathfrak{s}_{i+1}\lambda \left(
        {\color{black}\mathcal{F}^{n}_{{i+3/2}} (\mathfrak{s}_{i+1}u_{i+1}^{n},\mathfrak{s}_{i+1}u_{i+1}^{n})}
        -
          {\color{black}\mathcal{F}^{n}_{{i+1/2}}(\mathfrak{s}_{i+1}u_{i+1}^{n},\mathfrak{s}_{i+1}u_{i+1}^{n})}
      \right) \\\label{3}C_i&=+\lambda  \mathfrak{s}_i \left(
        {\color{black}\mathcal{F}^{n}_{{i+1/2}} (\mathfrak{s}_iu_{i}^{n},\mathfrak{s}_iu_{i}^{n})}
        -
          {\color{black}\mathcal{F}^{n}_{{i-1/2}} (\mathfrak{s}_iu_{i}^{n},\mathfrak{s}_iu_{i}^{n})}
      \right).
  \end{align}
   Using $0\le   \mathfrak{s}_{i}a^{n}_{{i-1/2}},\, \mathfrak{s}_{i}b^{n}_{{i+1/2}}\le 1/3$ {(cf.~\eqref{eq:bounda})}, we have 
    \begin{equation}\label{A}
        \sum_{i\in\Z}\modulo{A_i}\le\sum_{i\in\Z}\modulo{\Delta_+\mathfrak{s}_iu_i^{n}} .
    \end{equation}
   {Further,}
    \begin{align}\nonumber
   &\frac{B_i+C_i}{\lambda}
 \\&=-
      \mathfrak{s}_{i+1} \left(
        {\color{black}\mathcal{F}^{n}_{{i+3/2}} (\mathfrak{s}_{i+1}u^{n}_{i+1},\mathfrak{s}_{i+1}u^{n}_{i+1})}
        -
          {\color{black}\mathcal{F}^{n}_{{i+1/2}} (\mathfrak{s}_{i+1}u^{n}_{i+1},\mathfrak{s}_{i+1}u^{n}_{i+1})}
      \right) \nonumber\\
      &\quad+\mathfrak{s}_{i}\left(
        {\color{black}\mathcal{F}^{n}_{{i+1/2}} (\mathfrak{s}_{i}u^{n}_{i},\mathfrak{s}_{i}u^{n}_{i})}
        -
          {\color{black}\mathcal{F}^{n}_{{i-1/2}} (\mathfrak{s}_{i}u^{n}_{i},\mathfrak{s}_{i}u^{n}_{i})}
      \right)\nonumber\\
   &=   -   \mathfrak{s}_{i+1}f(\mathfrak{s}_{i+1}u^{n}_{i+1})\left(
 \nu(c^{n}_{{i+3/2}})-
         \nu(c^{n}_{{i+1/2}})
      \right) +
  \mathfrak{s}_{i}f(\mathfrak{s}_{i}u^{n}_{i} )\left(
 \nu(c^{n}_{{i+1/2}})-
         \nu(c^{n}_{{i-1/2}})
      \right)\nonumber\\
      &=  - \mathfrak{s}_{i+1}f(\mathfrak{s}_{i+1}u^{n}_{i+1})\left(
 \nu(c^{n}_{{i+3/2}})-
         \nu(c^{n}_{{i+1/2}})
      \right) +  \mathfrak{s}_{i}f(\mathfrak{s}_{i}u^{n}_{i})\left(
 \nu(c^{n}_{{i+3/2}})-
         \nu(c^{n}_{{i+1/2}})
      \right)\nonumber\\&\quad-   \mathfrak{s}_{i}f(\mathfrak{s}_{i}u^{n}_{i})\left(
 \nu(c^{n}_{{i+3/2}})-
         \nu(c^{n}_{{i+1/2}})
      \right)+
\mathfrak{s}_{i}f(\mathfrak{s}_{i}u^{n}_{i})\left(
 \nu(c^{n}_{{i+1/2}})-
         \nu(c^{n}_{{i-1/2}})
      \right)\nonumber\\
       &=  \nonumber  
      - (\mathfrak{s}_{i+1}f(\mathfrak{s}_{i+1}u^{n}_{i+1})- \mathfrak{s}_{i}f(\mathfrak{s}_{i}u^{n}_{i}))\left(
 \nu(c^{n}_{{i+3/2}})-
         \nu(c^{n}_{{i+1/2}})
      \right)\\
      \nonumber&\quad-  \mathfrak{s}_{i}f(\mathfrak{s}_{i}u^{n}_{i})\left(
 \nu(c^{n}_{{i+3/2}})-2
         \nu(c^{n}_{{i+1/2}})+
         \nu(c^{n}_{{i-1/2}})
      \right).
    \end{align}
Using { $f(0)=0,$} we have
    \begin{align*}
   |f(\mathfrak{s}_{i}u^{n}_{i})|
    & \leq  \abs{f}_{\lip(\R)}\norma{\mathfrak{s}}_{\L\infty(\R)}|u_i^n|,\\
\modulo{\mathfrak{s}_{i+1}f(u^{n}_{i+1}\mathfrak{s}_{i+1})-\mathfrak{s}_{i}f(\mathfrak{s}_{i}u^{n}_{i})}
&=\modulo{\mathfrak{s}_{i+1}f(u^{n}_{i+1}\mathfrak{s}_{i+1})-\mathfrak{s}_{i+1}f(u^{n}_{i}\mathfrak{s}_{i})+\mathfrak{s}_{i+1}f(u^{n}_{i}\mathfrak{s}_{i})-\mathfrak{s}_{i}f(\mathfrak{s}_{i}u^{n}_{i})}\\
&\le\modulo{f(u^{n}_{i+1}\mathfrak{s}_{i+1})-f(u^{n}_{i}\mathfrak{s}_{i})}\mathfrak{s}_{i+1}+f(u^{n}_{i}\mathfrak{s}_{i})\modulo{\mathfrak{s}_{i+1}-\mathfrak{s}_{i}}\\   &\le \abs{f}_{\lip(\R)}\modulo{\Delta_+(\mathfrak{s}_{i}u_{i}^{n})} \mathfrak{s}_{i+1}+\abs{f}_{\lip(\R)}\mathfrak{s}_{i}u_{i}^{n}\modulo{\mathfrak{s}_{i+1}-\mathfrak{s}_{i}}.
 \end{align*}Finally, using the above estimate in addition to Lemma~\ref{lem:AB} and { mean value theorem on $\nu$ and $\nu'$}, we have
\begin{align*}
  &|{B_i+C_i}|\\
  &\le\lambda \modulo{\mathfrak{s}_{i+1}f(\mathfrak{s}_{i+1}u^{n}_{i+1})-\mathfrak{s}_{i}f(\mathfrak{s}_{i}u^{n}_{i})}{\norma{\nu'}_{\L\infty(\R)}}\modulo{c^{n}_{{i+3/2}}-
         c^{n}_{{i+1/2}}}\\
         &\quad +
      \lambda \mathfrak{s}_{i}\abs{f}_{\lip(\R)}\norma{\mathfrak{s}}_{\L\infty(\R)}|u_i^n|{\norma{\nu'}_{\L\infty(\R)}}\modulo{c^{n}_{{i+3/2}}-
         2c^{n}_{{i+1/2}}+c^{n}_{{i-1/2}}}\\
         &{\quad+  \lambda \mathfrak{s}_{i}\abs{f}_{\lip(\R)}\norma{\mathfrak{s}}_{\L\infty(\R)}|u_i^n|\abs{\nu'}_{\lipR}{\abs{c^{n}_{i+1/2}-c^{n}_{i-1/2}}}\left({\abs{c^{n}_{i+1/2}-c^{n}_{i-1/2}}}+{\abs{c^{n}_{i+3/2}-c^{n}_{i+1/2}}}\right)}
        \\
         &\le  \lambda \norma{\nu'}_{\L\infty(\R)}(\abs{f}_{\lip(\R)}\modulo{\Delta_+(\mathfrak{s}_{i}u_{i}^{n})} \mathfrak{s}_{i+1}) \mathcal{K}_1\Delta x    \\&\quad +\lambda {\norma{\nu'}_{\L\infty(\R)}}(\abs{f}_{\lip(\R)}\mathfrak{s}_{i}u_{i}^{n}\modulo{\mathfrak{s}_{i+1}-\mathfrak{s}_{i}} ) \mathcal{K}_1\Delta x\\&\quad +
       \lambda \mathfrak{s}_{i}\abs{f}_{\lip(\R)}\norma{\mathfrak{s}}_{\L\infty(\R)}|u_i^n|{\norma{\nu'}_{\L\infty(\R)}} \mathcal{K}_2\Delta x^2\\
         &{\quad+  2\lambda \mathfrak{s}_{i}\abs{f}_{\lip(\R)}\norma{\mathfrak{s}}_{\L\infty(\R)}|u_i^n|\abs{\nu'}_{\lipR}\mathcal{K}^2_1\Delta x^2.}   \end{align*}
Now, summing up  $|{B_i+C_i}|$ over all $i$,
        \begin{align}
   \nn \sum_{i\in\Z} |{B_i+C_i}| &\le \mathcal{K}_1\Delta t {\norma{\mathfrak{s}}}_{\L\infty({\mathbb{R}})} \abs{f}_{\lip(\R)}  { \norma{\nu'}_{\L\infty(\R)}}\sum_{i\in\Z}\abs{\Delta_+(\mathfrak{s_i}u_i^{n})}\\ \nn 
   &\quad+\mathcal{K}_1\Delta t     {\abs{\mathfrak{s}}_{\BV(\R)}}\abs{f}_{\lip(\R)}  { \norma{\nu'}_{\L\infty(\R)}}\norma{\mathfrak{s}}_{\L\infty({\mathbb{R}})}\norma{u^{n}}_{\L\infty}\\ \nn 
      &\quad+ \mathcal{K}_2\Delta t {\norma{\mathfrak{s}}}_{\L\infty({\mathbb{R}})}\abs{f}_{\lip(\R)}{\norma{\nu'}_{\L\infty(\R)}} \norma{\mathfrak{s}u^{n}}_{\L1}\\ \nn
      &{\quad+ 2\mathcal{K}^2_1\Delta t\norma{\mathfrak{s}}_{\L\infty(\R)}\abs{f}_{\lip(\R)}\abs{\nu'}_{\lipR}\norma{\mathfrak{s}u^{n}}_{\L1} }\\  \label{BC}  
       &=\mathcal{K}_5\Delta t     \sum_{i\in\Z}|\Delta_+(\mathfrak{s_i}u_i^{n})| 
       + \mathcal{K}_6\Delta t,
       \end{align}
      where  \begin{align*}
\mathcal{K}_5&=\mathcal{K}_1{\norma{\mathfrak{s}}}_{\L\infty({\mathbb{R}})} \abs{f}_{\lip(\R)}  { \norma{\nu'}_{\L\infty(\R)}},\\
\mathcal{K}_6&=\mathcal{K}_1{\abs{\mathfrak{s}}_{\BV(\R)}}\abs{f}_{\lip(\R)}  { \norma{\nu'}_{\L\infty(\R)}}\norma{\mathfrak{s}}_{\L\infty({\mathbb{R}})}\norma{u^{n}}_{\L\infty}+\mathcal{K}_2{\norma{\mathfrak{s}}}_{\L\infty({\mathbb{R}})}\abs{f}_{\lip(\R)}{\norma{\nu'}_{\L\infty(\R)}} \norma{\mathfrak{s}u^{n}}_{\L1}\\
&\quad{+2\mathcal{K}^2_1\norma{\mathfrak{s}}_{\L\infty(\R)}\abs{f}_{\lip(\R)}\abs{\nu'}_{\lipR}\mathcal{K}^2_1\norma{\mathfrak{s}u^{n}}_{\L1}}.
      \end{align*}
        { Combining the estimates \eqref{ABC}, \eqref{A} and \eqref{BC}, we get}
    {    \begin{align*}
 \sum_{i\in\Z}\abs{\Delta_+(\mathfrak{s}_{i}u_{i}^{n+1})}&\le  \sum_{i\in\Z}\abs{\Delta_+(\mathfrak{s_i}u_i^{n})}   (1+\mathcal{K}_5\Delta t  )  
     + \mathcal{K}_6\Delta t\\
&\leq\sum_{i\in\Z}\abs{\Delta_+(\mathfrak{s_i}u_i^{n-1})}   (1+\mathcal{K}_5\Delta t  )^2  
     + \mathcal{K}_6\Delta t\left(1+(1+\mathcal{K}_5\Delta t  ) \right) 
     \\&\le (\sum_{i\in\Z}\abs{\Delta_+(\mathfrak{s_i}u_i^{n-2})}   (1+\mathcal{K}_5\Delta t  )  
     + \mathcal{K}_6\Delta t)  (1+\mathcal{K}_5\Delta t  )^2  
     + \mathcal{K}_6\Delta t\left(1+(1+\mathcal{K}_5\Delta t  ) \right) 
\\&=\sum_{i\in\Z}\abs{\Delta_+(\mathfrak{s_i}u_i^{n-2})}   (1+\mathcal{K}_5\Delta t  )^3   
     + \mathcal{K}_6\Delta t\sum_{k=0}^{2}\left(1+\mathcal{K}_5\Delta t\right)^k.
        \end{align*}}
           {Proceeding like this, we have,
        \begin{align}\nn
    \sum_{i\in\Z}\abs{\Delta_+(\mathfrak{s}_{i}u_{i}^{n})}&\le  \sum_{i\in\Z}\abs{\Delta_+(\mathfrak{s_i}u_i^{0})}   (1+\mathcal{K}_5\Delta t  )^{n} 
     + \mathcal{K}_6\Delta t\sum_{k=0}^{n}\left(1+\mathcal{K}_5\Delta t\right)^k\\ \nn 
     &\le \sum_{i\in\Z}\abs{\Delta_+(\mathfrak{s_i}u_i^{0})}(1+\mathcal{K}_5\Delta t)^{n}
     + \mathcal{K}_6\Delta t\frac{1-(1+\mathcal{K}_5\Delta t)^{n}}{1-(1+\mathcal{K}_5\Delta t)}\\ \nn 
     &= \sum_{i\in\Z}\abs{\Delta_+(\mathfrak{s_i}u_i^{0})}(1+\mathcal{K}_5\Delta t  )^{n}
     + \frac{\mathcal{K}_6}{\mathcal{K}_5}((1+\mathcal{K}_5\Delta t)^{n}-1)\\ \label{bv:su}
     &\le \sum_{i\in\Z}\abs{\Delta_+(\mathfrak{s_i}u_i^{0})}\exp(\mathcal{K}_5 t^n)
     + \frac{\mathcal{K}_6}{\mathcal{K}_5}(\exp(\mathcal{K}_5 t^n)-1).
        \end{align}}
{Finally, note that \begin{align*}
        \abs{\Delta_+u_i^n} \leq \frac{1}{\inf\limits_{x\in \R}s(x)} \left( \abs{\Delta_+(\mathfrak{s}_{i}u_i^{n})}+
 \abs{u_i^n}\abs{\Delta_+\mathfrak{s}_{i}}\right),
 \end{align*}
 consequently, we have
 \begin{align}\label{bv:u1}
     \sum\limits_{i\in \Z }\abs{\Delta_+u_i^n}\leq \frac{1}{\inf\limits_{x\in \R}s(x)} \left( \sum\limits_{i\in \Z } \abs{\Delta_+(\mathfrak{s}_{i}u_{i}^{n})}+ \mathcal{K}_4\exp(\mathcal{K}_3 \, t ) \,
    \norma{{u^{\Delta}(0)}}_{\L\infty(\R)}
\abs{\mathfrak{s}}_{\BV(\R)}\right).
 \end{align} Now, the lemma follows by substituting the BV estimate \eqref{bv:su} in \eqref{bv:u1}.}
      \end{proof}
\begin{lemma}[Time estimate]\label{lem:L1t}
For $m>n {\in \{0, \ldots, N\}},$ we have the following time estimate:
\begin{eqnarray*}
\Delta x\sum\limits_{i\in \mathbb{Z}} \abs{u_i^{m}-u_i^{n}} \leq \mathcal{K}_7 \Delta t (m-n),
\end{eqnarray*}
where $\mathcal{K}_7$ depends on 
  $\sum_{i\in\mathbb{Z}}
        \modulo{\Delta_+\mathfrak{s}_{i}u_i^n} ,{\norma{\nu'}_{\L\infty(\R)}},\norma{\mathfrak{s}u^n}_{\L\infty}, \abs{\mathfrak{s}}_{\BV(\R)},  \norma{u_0}_{\L1(\R)}$ and $\mu$.
\end{lemma}
\begin{proof}
  Owing to the total variation bound of $\mathfrak{s}u^n$, and following~\cite[Lemma~2.7]{ACT2015}, the proof follows.\end{proof}
{To show that the limit of the numerical approximations {$u^{\D}$} are indeed the entropy solutions,} we prove that the approximate solutions {$u^{\D}$} satisfy a discrete
entropy condition. For this, we introduce the
following numerical entropy flux,
\begin{align*}
{G^{n}_{{i+1/2}}}{(p,q,\alpha)}&=\mathcal{F}^{n}_{{i+1/2}}\left(\mathfrak{s}_i{\max(p, \mathfrak{s}_{\alpha}(x_i))},\,\mathfrak{s}_{i+1}{\max(q,\mathfrak{s}_{\alpha}(x_{i+1}))}\right)\\&\quad
  -
  \mathcal{F}^{n}_{{i+1/2}}\left(\mathfrak{s}_i{\min(p, \mathfrak{s}_{\alpha}(x_i))},\, \mathfrak{s}_{i+1}{\min(q, \mathfrak{s}_{\alpha}(x_{i+1}))}\right),
\end{align*}  
 where {$s_{\alpha}$ is as in def.~\ref{def:entropy}}.
 {
 \begin{remark}\normalfont
It is straightforward to check that,
 \begin{align*}
{G^{n}_{{i+1/2}}} (p,p,\alpha) & = \nu(c^n_{{i+1/2}}) \sgn(p-\mathfrak{s}_{\alpha}(x_i)) \left(f(\mathfrak{s}_{i}p)-f(\alpha) \right)\\ 
  & = \nu(c^n_{{i+1/2}}) \sgn(\mathfrak{s}_{i}p-\alpha) \left(f(\mathfrak{s}_{i}p)-f(\alpha) \right)\\ &\approx \nu(\mu*\overline{\nu}(u))(t^n,x_{i+1/2}) \sgn(\mathfrak{s}_{i}p-\alpha) \left(f(\mathfrak{s}_{i}p)-f(\alpha) \right),
 \end{align*} which shows that indeed the numerical entropy flux is consistent with the adapted entropy flux (see def.~\ref{def:entropy}).
 \end{remark}}
\begin{lemma}[Discrete entropy condition]
  \label{lem:entropy}
  Let assumptions~\textbf{(H1)}, \textbf{(H2)} and conditions~\eqref{CFl} hold. Fix an initial datum
  $u_0 \in (\L1  \cap \BV) (\R;{[0,\infty)})$. Then,
  the approximate solution ${u^{\Delta}}$ defined by the {marching formula}~\eqref{scheme2}
  satisfies for all $i \in \Z, {0\leq n \leq N}$ and
  for all $\alpha\in \R,$ the discrete entropy inequality
 \begin{align} 
 \begin{split}
\label{eq:discrete_entropy}
 &\modulo{u_i^{n+1}-\mathfrak{s}_{\alpha}(x_i)}-\modulo{u_i^{n}- \mathfrak{s}_{\alpha}(x_i)}+\lambda\left({{G^{n}_{{i+1/2}}}(u_i^{n} ,u_{i+1}^{n},\alpha)}-G^{n}_{{i-1/2}}(u_{i-1}^{n} ,u_i^{n},\alpha)\right)\\
 &\quad \quad +\lambda\sgn(u_i^{n+1}-\mathfrak{s}_{\alpha}(x_i))f(\alpha)(\nu(c_{i+1/2}^{n})-\nu(c_{i-1/2}^{n}))\le 0.
 \end{split}
 \end{align}
 \end{lemma}

\begin{proof}
 {Let $i\in \Z,$ $n  \in  \{0, \ldots, N\}$}
and consider, 
\begin{align*}
&\abs{u_i^{n}-\mathfrak{s}_{\alpha}(x_i)}-\lambda\left({G^{n}_{{i+1/2}}}{(u_i^n,u_{i+1}^n,\alpha)}-G^{n}_{{i-1/2}}{(u_{i-1}^n,u_{i}^n,\alpha)}\right)\\
&=H(\nu(c^{n}_{{i-1/2}}),\nu(c^{n}_{{i+1/2}}),\max(u_{i-1}^{n},s_{\alpha}(x_{i-1})),\max(u_i^{n},s_{\alpha}(x_{i})),\max(u_{i+1}^{n},s_{\alpha}(x_{i+1})))\\&\quad-H(\nu(c^{n}_{{i-1/2}}),\nu(c^{n}_{{i+1/2}}),\min(u_{i-1}^{n},\mathfrak{s}_{\alpha}(x_{i-1})),\min(u_{i}^{n},\mathfrak{s}_{\alpha}(x_{i})),\min(u_{i+1}^{n},\mathfrak{s}_{\alpha}(x_{i+1})))\\
&\geq \max(u_i^{n+1},H(\nu(c^{n}_{{i-1/2}}),\nu(c^{n}_{{i+1/2}}),\mathfrak{s}_{\alpha}(x_{i-1}),\mathfrak{s}_{\alpha}(x_{i}),\mathfrak{s}_{\alpha}(x_{i+1})))\\&\quad-\min(u_i^{n+1},H(\nu(c^{n}_{{i-1/2}}),\nu(c^{n}_{{i+1/2}}),\mathfrak{s}_{\alpha}(x_{i-1}),\mathfrak{s}_{\alpha}(x_{i}),\mathfrak{s}_{\alpha}(x_{i+1}))
\\
&=\abs{u_i^{n+1}-H(\nu(c^{n}_{{i-1/2}}),\nu(c^{n}_{{i+1/2}}),\mathfrak{s}_{\alpha}(x_{i-1}),\mathfrak{s}_{\alpha}(x_{i}),\mathfrak{s}_{\alpha}(x_{i+1}))}\\
&=\sgn(u_i^{n+1}-H(\nu(c^{n}_{{i-1/2}}),\nu(c^{n}_{{i+1/2}}),\mathfrak{s}_{\alpha}(x_{i-1}),\mathfrak{s}_{\alpha}(x_{i}),\mathfrak{s}_{\alpha}(x_{i+1})))\\
&\quad \times(u_i^{n+1}-H(\nu(c^{n}_{{i-1/2}}),\nu(c^{n}_{{i+1/2}}),\mathfrak{s}_{\alpha}(x_{i-1}),\mathfrak{s}_{\alpha}(x_{i}),\mathfrak{s}_{\alpha}(x_{i+1})))\\&
= \sgn\left(u_i^{n+1}-[\mathfrak{s}_{\alpha}(x_i)-\lambda f(\alpha) (\nu(c^{n}_{{i+1/2}})-\nu(c^{n}_{{i-1/2}}))] \right)  \\& \quad \times \left(u_i^{n+1}-[\mathfrak{s}_{\alpha}(x_i)-\lambda f(\alpha) (\nu(c^{n}_{{i+1/2}})-\nu(c^{n}_{{i-1/2}}))] \right)
\\
&\geq \sgn \left(u_i^{n+1}-\mathfrak{s}_{\alpha}(x_i) \right) \times \left(u_i^{n+1}-[\mathfrak{s}_{\alpha}(x_i)-\lambda f(\alpha) (\nu(c^{n}_{{i+1/2}})-\nu(c^{n}_{{i-1/2}}))] \right) \\&= \abs{u_i^{n+1}-\mathfrak{s}_{\alpha}(x_i)} + \lambda f(\alpha) \sgn (u_i^{n+1}-\mathfrak{s}_{\alpha}(x_i))[\nu(c^{n}_{{i+1/2}})-\nu(c^{n}_{{i-1/2}})].
\end{align*} 
Here, the first inequality follows from the monotonicity of the scheme (cf. Remark \ref{rem:mono}) and the last inequality follows from the inequality
$
\sgn(a+b)(a+b) \geq \sgn(a)(a+b).
$ This completes the proof.
\end{proof}
{      
\begin{theorem}[Convergence]\label{convergence}
As $\Delta x \rightarrow 0$, the approximations ${u^{\Delta}}$ 
generated by the  marching formula \eqref{scheme2} converge in $L^1_{\loc}(\overline{Q}_T)$ and pointwise a.e.~in $\overline{Q}_T$ to the entropy solution 
	$u$ of the Cauchy problem \eqref{eq:u}-\eqref{eq:id}
  with initial data $u_0 \in (L^1 \cap \BV)(\R)$.
\end{theorem}
\begin{proof}
Lemma \ref{lem:BV} implies that the sequence of functions $u^{\D}(t,\cdot)$ is uniformly total variation bounded. Owing to the time estimate (cf.~Lemmma \ref{lem:L1t}) and Helly's theorem (see \cite[Cor.~A.10]{HR2015}) there exists $u\in \L\infty([0,T]; \BV(\R)) \cap C([0,T];L^1(\R))$ such that up to a subsequence ${u^{\Delta}} \rightarrow u$ in $L^1_{loc}(Q_T).$ 
An amalgamation of the Lax-Wendroff type argument presented in  \cite[Thm.~5.1]{BBKT2011} and \cite[Lem.~4.4]{GJT2020} implies that the limit $u$ indeed satisfies the entropy condition def.~\ref{def:entropy}. We briefly sketch the proof for the sake of completeness.

For any non negative test function 
 $\phi \in C_c^{\infty}(\R),$ define a grid function as $\phi_i^n:=\phi(t^n,x_i).$ Multiplying the discrete entropy inequality \eqref{eq:discrete_entropy} by $\phi_i^n\D x,$ summing over $i\in \Z, 0\leq n \leq N$ and finally applying summation by parts, one gets:

\begin{align*}
   0&\leq  \D t \sum\limits_{n=0}^N \sum\limits_{i\in \Z} \abs{u_i^n-\mathfrak{s}_{\alpha}(x_i)} 
  \frac{\phi_i^{n+1}-\phi_i^n}{\D t }\\
   &\quad + \D x \D t \sum\limits_{n=0}^N \sum\limits_{i\in \Z} {{G^{n}_{{i+1/2}}}(u_i^{n} ,u_{i+1}^{n},\alpha)} \frac{\phi_{i+1}^{n}-\phi_i^{n}}{\D x }\\
   &\quad - \D x \D t \sum\limits_{n=0}^N \sum\limits_{i\in \Z} \lambda\sgn(u_i^{n+1}-\mathfrak{s}_{\alpha}(x_i))f(\alpha)\frac{(\nu(c_{i+1/2}^{n})-\nu(c_{i-1/2}^{n}))}{\D x}\phi_i^n\\
   &\quad + \D x \sum\limits_{i\in \Z} \abs{u_i^0-\mathfrak{s}_{\alpha}(x_i)} \phi_i^0
   \\
   &:=E^{\D}_1+E^{\D}_2+E^{\D}_3+E^{\D}_4.
   \end{align*}
   Define 
\begin{align*}
\tilde{E}^{\D}_2&:= \D x \D t \sum\limits_{n=0}^N \sum\limits_{i\in \Z} G^{n}_{i+1/2}(u_i^{n} ,u_{i}^{n} ) \frac{\phi_{i+1}^{n}-\phi_i^{n}}{\D x },\\ 
\tilde{E}^{\D}_3&:= \D x \D t \sum\limits_{n=0}^N \sum\limits_{i\in \Z} \lambda\sgn(u_i^{n}-\mathfrak{s}_{\alpha}(x_i))f(\alpha)\frac{(\nu(c_{i+1/2}^{n})-\nu(c_{i-1/2}^{n}))}{\D x}\phi_i^n.
\end{align*}
Using the estimates \eqref{eq:c_i+1/2}, \begin{align*}
    \abs{G^n_{i+1/2}(u_i^n,u_{i+1}^n,\alpha)-G^n_{i+1/2}(u_i^n,u_i^n,\alpha)} \leq \abs{v(c^n_{i+1/2})} \abs{f}_{\lip(\R)}\abs{u_{i+1}^n-u_{i}^n},\end{align*} and following the arguments in \cite[Thm.~5.1]{BGKT2008},
one obtains\begin{align*}
\abs{E_2^{\D}-\tilde{E}_2^{\D}} \leq C \D x, \text{ and }
\abs{E_3^{\D}-\tilde{E}_3^{\D}} \leq C \D x.
\end{align*}
   Now, invoking Lebesgue dominated convergence theorem, as $\D x \rightarrow 0$, we have:
  \begin{align*}
       E^{\D}_1 &\rightarrow \int_{\Pi_T} \abs{u(t,x)-\mathfrak{s}_{\alpha}(x)}\phi_t(t,x) \d t \d x, \\
\tilde{E}^{\D}_2 &\rightarrow \int_{\Pi_T} \sgn(u(t,x)-k)(f(s(x)u(t,x))-\mathfrak{s}_{\alpha}(x)) (\mu*\overline{\nu}(u))(t,x)\phi_x(t,x) \d t \d x,\\
\tilde{E}^{\D}_3 &\rightarrow \int_{\Pi_T} \sgn(u(t,x)-k)f(\alpha) \partial_x(\mu*\overline{\nu}(u))(t,x)\phi(t,x) \d t \d x,\\
       E^{\D}_4 &\rightarrow \int_{\R} \abs{u_0(x)-\mathfrak{s}_{\alpha}(x)} \phi(0,x) \d x.
    \end{align*}
Finally, invoking the uniqueness of the entropy solution (see Thm.~ \ref{uniqueness}), in fact, the entire sequence $u^{\Delta}$ converges to the entropy solution $u$.
This completes the proof.
\end{proof}

}
{\section{Numerical Experiments}\label{NR}
\begin{figure}[h!]
 \centering
 \begin{subfigure}{.45\textwidth}
\includegraphics[width=\textwidth,keepaspectratio]{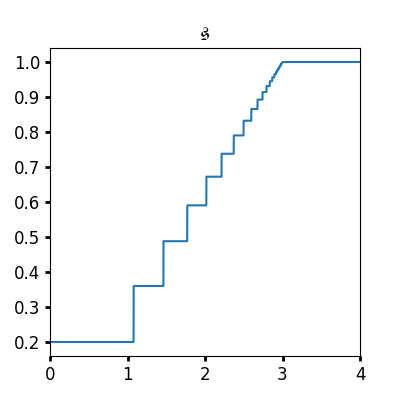}
\end{subfigure}
\hfill
\begin{subfigure}{.45\textwidth}
\includegraphics[width=\textwidth,keepaspectratio]{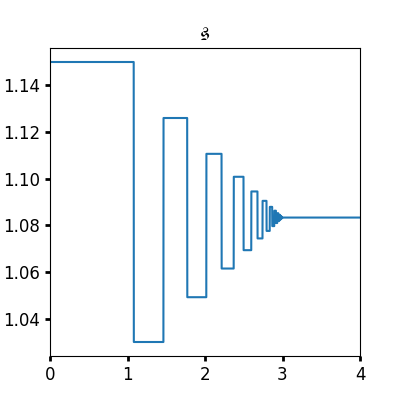}
\end{subfigure}\hspace{1cm}
\begin{subfigure}{.45\textwidth}
\includegraphics[width=\textwidth,keepaspectratio]{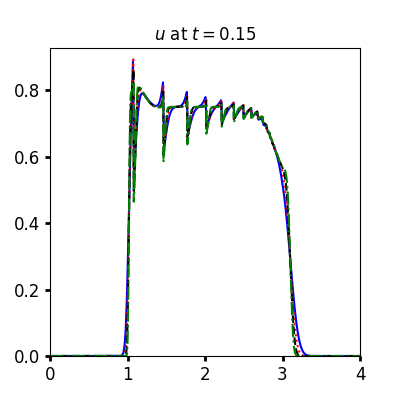}
\end{subfigure}
\hfill
\begin{subfigure}{.45\textwidth}
\includegraphics[width=\textwidth,keepaspectratio]{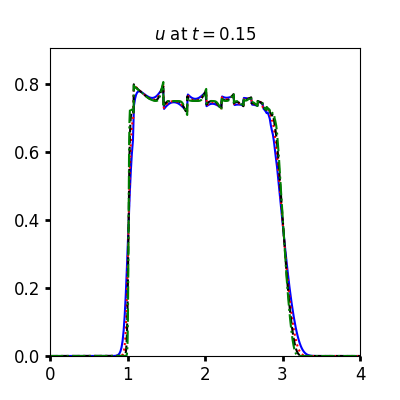}
\end{subfigure}\hspace{1cm}
\begin{subfigure}{.45\textwidth}
\includegraphics[width=\textwidth,keepaspectratio]{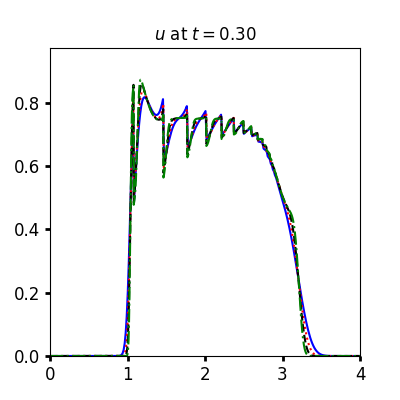}
 \caption{$f(u)=u$, monotone $\mathfrak{s}$}
 \label{linear}
\end{subfigure}
\hfill
\begin{subfigure}{.45\textwidth}
\includegraphics[width=\textwidth,keepaspectratio]{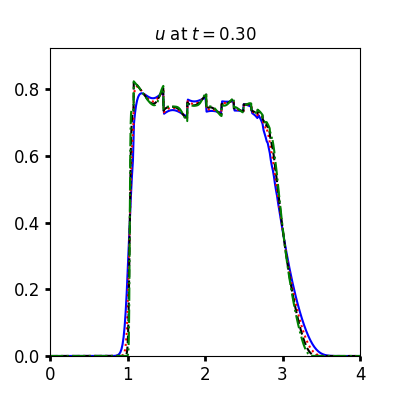}
 \caption{$f(u)=u(1-u)$, non-monotone $\mathfrak{s}$}
 \label{nonlinear}
\end{subfigure}

\caption{Solution to \eqref{eq:u}-\eqref{eq:id} on the domain $[0, \, 4]$ at times $t =
    0.00,\; 0.15,\;0.30$ with decreasing mesh size  $\Delta x=1/75 $({\color{blue}\full}), $\Delta x=1/150 $({\color{magenta}\dotted}), $\Delta x=1/300 $({\color{black}\dashed}) and $\Delta x=1/600 $({\color{darkgreen}\chainn})}
  \label{fig:ex1}
\end{figure} 
We present some numerical experiments to illustrate the theory presented in the previous sections. We show the results for \eqref{scheme2} with Lax--Friedrichs type flux. The results obtained by \eqref{scheme2} with Godunov type flux are similar, and are not shown here. Throughout the section, $\Theta$ is chosen to be $1/3$, and $\lambda$ is chosen so as to satisfy the CFL condition \eqref{CFl}. We consider the \eqref{eq:u}-\eqref{eq:id} with the kernel
  \begin{align*}
\mu(x)&=L(\epsilon^2-x^2)^{3}\mathbbm{1}_{(-\epsilon,\epsilon)}(x), \,x\in \R,
\end{align*} $L$ being such that $\int_{\R}\mu(x)\d{x}=1,$ the velocity 
 $\nu(a)=1-a,$ and with 
  \begin{align*}\label{s}
\quad \mathfrak{s}(x)= 
\begin{cases}
\frac{1}{3}a_1, \quad & x \le a_1,\\
\frac{1}{3}\sum\limits_{n=2}^\infty a_n \mathbbm{1}_{[a_n,a_{n+1}]}(x), \quad & x \in (a_1,3],\\
1, \quad & x> 3,\\
\end{cases} ,
\end{align*}
admitting infinitely many 
\begin{figure}[h!]
\noindent\centering\begin{minipage}{0.9\textwidth}
\includegraphics{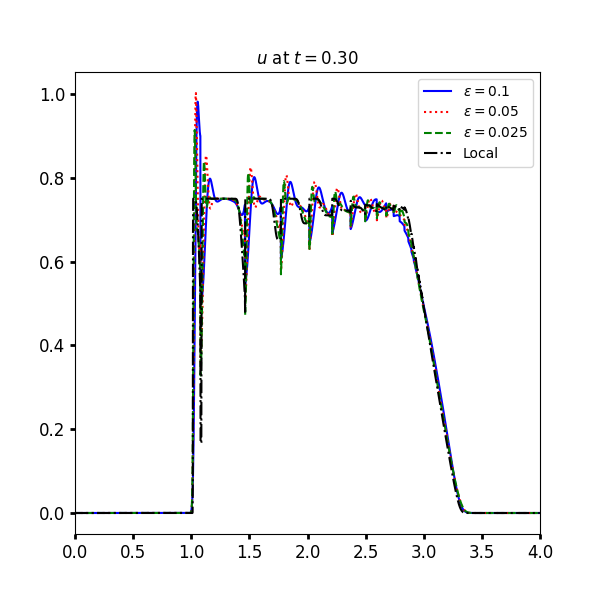}
  \end{minipage} \noindent\centering\begin{minipage}{0.9\textwidth}
    \centering
\begin{tabular}{|c|c|c|c|c|c|c|c|c|c|}\hline
\textbf{$\epsilon$} & $0.1$ & $0.05$ & $0.0.025$ &$0.01$ \\ 
     \hline
    \textbf{$e_{\epsilon}$} & $0.1031$ & $0.0685$ & $0.0390$&$0.0257$ \\
     \hline
\end{tabular}
  \end{minipage} \caption{$\L1$ difference between Nonlocal solution of \eqref{eq:u}-\eqref{eq:id}  with decreasing kernel support $\epsilon$ and corresponding local solution, $\Delta x=1/600,T=0.3$: linear $f$, and increasing $\mathfrak{s}$}\label{fig:my_label21}
\end{figure}
\begin{figure}[h!]\centering
\includegraphics[width=.49\textwidth,keepaspectratio]{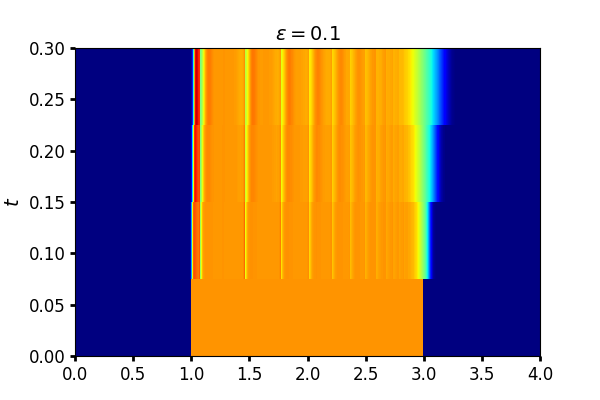}
\includegraphics[width=.49\textwidth,keepaspectratio]{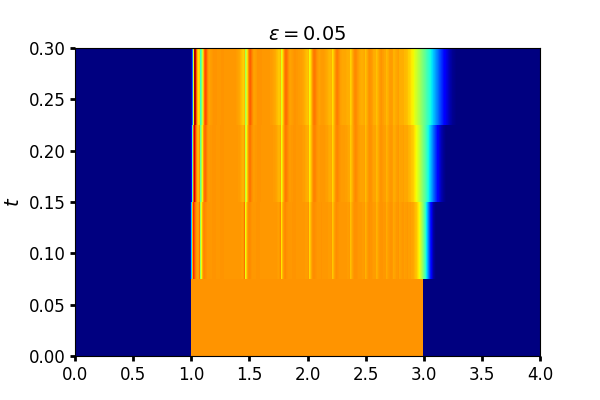}\\
\includegraphics[width=.49\textwidth,keepaspectratio]{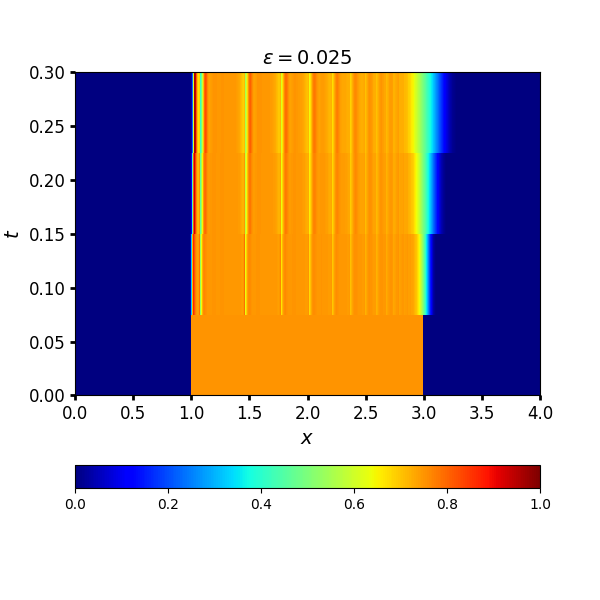}
\includegraphics[width=.49\textwidth,keepaspectratio]{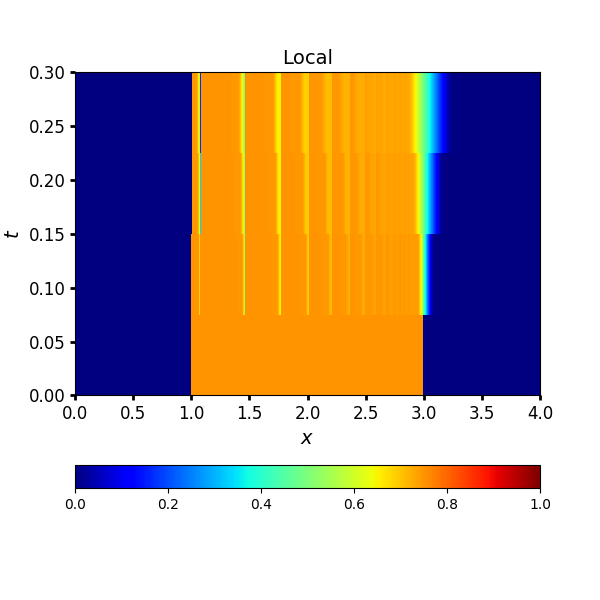}
\caption{Nonlocal solution of \eqref{eq:u}-\eqref{eq:id} with decreasing kernel support $\epsilon$ and corresponding local solution, $\Delta x=1/600, t\in[0,0.3]$: linear $f$, and increasing $\mathfrak{s}$}
  \label{fig:ex2}
\end{figure}discontinuities at the points $\{a_n\}_{n\in \N},$ with 
\begin{enumerate}
    \item  $a_n=3(1-0.8^n)$ making $\mathfrak{s}$ non--decreasing
    \item  $a_n=(1-(-0.8)^n)/12+1$ making $\mathfrak{s}$ non--monotonic.
\end{enumerate} 
We also present numerical integration with two choices of local part of the fluxes namely, $f(u)=u$ and $f(u)=u(1-u)$. All the parameters described above fit in the hypothesis of the paper. We choose the domain $[0, \, 4]$ and the time
 interval $[0, \, 0.3]$ with initial data
\begin{equation*}
    \label{eq:ex1} u_0(x)=0.75\mathbbm{1}_{(1,3)}(x).\end{equation*}

Figure \ref{fig:ex1} displays the numerical approximations generated by the numerical  
scheme \eqref{scheme2}, 
with decreasing grid size $\Delta x$, starting with $\Delta x =1/75$ for various combinations of the above parameters at times $t=0.15$ and $t=0.30$. 

Figure \ref{linear} represents the behavior of the solutions with monotonic $\mathfrak{s}$ and with linear $f(u)=u.$ It can be seen that the density $u$ goes beyond $0.75$ {violating the  invariant region principle}, a phenomenon also observed in non-local conservation laws with smooth local part, see \cite{ACT2015}. 

Figure \ref{nonlinear} represents the behavior of the solutions with non-monotonic $\mathfrak{s}$ and with non-linear $f(u)=u(1-u),$ with zeros at $u=0$ and $u=1.$ It can be seen that the density $u$ { always lies within the interval $[0,1]$ demonstrating the invariant region principle (see Lemma \ref{lem:invariant})}. 

It can be seen that the numerical scheme is able to capture the features of the solutions well. {The roughness of the coefficient $\mathfrak{s}$ leads to the abruptness in the entropy solution which can be seen along the discontinuities of $\mathfrak{s}$}.
 The results of other combinations of parameters(non-monotone $\mathfrak{s}$ with linear $f$, and non-linear $f$ with monotone $\mathfrak{s}$), are similar and not shown here.
  
Figures \ref{fig:my_label21}-\ref{fig:ex2} investigate 
the behavior of the approximate solutions as the kernel support $\epsilon\rightarrow0.$ For this experiment, we choose monotonic $\mathfrak{s}$ with $f(u)=u.$ With $\Delta x=1/600,$ we compare the nonlocal solution  ${u^{\Delta}_{\epsilon}}$ obtained by \eqref{scheme2} with the Lax-Friedrichs type flux with ${u^{\Delta}_{loc}}$, the solution for the corresponding local conservation law {\eqref{eq:u}-\eqref{eq:id}} obtained using the local Godunov scheme of \cite{GTV2022} at time $T=0.3$, as $\epsilon$ takes values $0.1, 0.05$, $0.025$ and $0.01$. 

Let $e_{\epsilon}:={\norma{u^{\Delta}_{loc}(0.3,\cdot)-u^{\Delta}_{\epsilon}(0.3,\cdot)}_{L^1(\R)}}.$ It can be seen in Figure \ref{fig:my_label21} that $e_{\epsilon}$ decreases as $\epsilon\rightarrow0^+$ and that the non-local approximations seem to converge to local solution with decreasing $\epsilon.$ The findings of the experiment suggest that one can possibly obtain the solutions of the local conservation laws, by decreasing the kernel support, and is a question of independent interest.}
\section*{Acknowledgement}
Parts of this work were carried out during GV's tenure of the ERCIM ‘Alain Bensoussan’ Fellowship Programme, and were supported in part by the project \textit{IMod --- Partial differential equations, statistics and data:
An interdisciplinary approach to data-based modelling}, project number 325114, from the Research Council of Norway. Parts of this work were also carried out during GV's tenure at IIM Indore and AA's research visit at NTNU, and were partially supported by AA's Seed Money Grant 
SM/11/2021/22 and by AA's Faculty Development Allowance, from IIM Indore.
\bibliographystyle{abbrv}
\bibliography{AHV_system_AHV}
\end{document}